\documentclass[reqno,a4paper]{amsart}
\usepackage{eucal,amsfonts,amssymb,amsmath,amsthm,epsfig,mathrsfs}

\usepackage{amscd,amsxtra}
\usepackage{enumerate}
\usepackage{latexsym}

\allowdisplaybreaks

 \makeatletter \@addtoreset{equation}{section}

\makeatother \makeatletter

\newtheorem{thm}{Theorem}[section]
\newtheorem{hyp}[thm]{Hypotheses}{\rm}
{\rm}
\newtheorem{lemm}[thm]{Lemma}
\newtheorem{coro}[thm]{Corollary}
\newtheorem{prop}[thm]{Proposition}
\newtheorem{defi}[thm]{Definition}
\newtheorem{rmk}[thm]{Remark}{\rm}

\newcommand{\R}{{\mathbb R}}
\newcommand{\N}{{\mathbb N}}

\newcommand{\Rd}{\mathbb R^d}

\newcommand{\bd}{\begin{defi}}
\newcommand{\ed}{\end{defi}}

\newcommand{\be}{\begin{equation}}
\newcommand{\ee}{\end{equation}}
\newcommand{\barr}{\begin{array}}
\newcommand{\earr}{\end{array}}
\newcommand{\bmn}{\begin{eqnarray}}
\newcommand{\emn}{\end{eqnarray}}
\newcommand{\bnm}{\begin{eqnarray*}}
\newcommand{\enm}{\end{eqnarray*}}
\newcommand{\bln}{\begin{subequations}}
\newcommand{\eln}{\end{subequations}}

\newcommand{\ba}{\begin{align}}
\newcommand{\ea}{\end{align}}
\newcommand{\banm}{\begin{align*}}
\newcommand{\eanm}{\end{align*}}
\newcommand{\one}{\mbox{$1\!\!\!\;\mathrm{l}$}}

\title[Hypercontractivity and asymptotic behaviour]{Hypercontractivity and asymptotic behaviour in nonautonomous Kolmogorov equations}
\author{L. Angiuli, L. Lorenzi and A. Lunardi}
\address{Dipartimento di Matematica, Universit\`a degli Studi di Parma, Parco Area delle Scienze 53/A, I-43124 Parma, Italy.}
\email{(luciana.angiuli, luca.lorenzi, alessandra.lunardi)@unipr.it}

\date{\today}
\begin{document}

\begin{abstract}
We consider a class of nonautonomous second order parabolic equations with unbounded coefficients defined in $I\times\R^d$, where $I$ is a right-halfline. We prove
logarithmic Sobolev and Poincar\'e inequalities with respect to an associated evolution system of measures
$\{\mu_t: t \in I\}$,  and we deduce hypercontractivity and asymptotic behaviour
results for the evolution operator $G(t,s)$. \end{abstract}

\maketitle
\section{Introduction}

We consider   nonautonomous Cauchy problems,
\begin{equation}\label{p_e}
\left\{
\begin{array}{ll}
D_tu(t,x)={\mathcal{A}}(t)u(t,x),\quad\quad & (t,x)\in (s,+\infty)\times \Rd,\\[1mm]
u(s,x)= f(x),\quad\quad & x\in \Rd,
\end{array}\right.
\end{equation}
where $\{\mathcal{A}(t)\}_{t\in I}$ is a family of second order differential operators,
\begin{equation}\label{a(t)}
(\mathcal{A}(t)\zeta)(x)= {\textrm{Tr}}(Q(t)D^2\zeta(x))+ \langle b(t,x), \nabla \zeta(x)\rangle,
\end{equation}
with smooth enough coefficients $Q=[q_{ij}]_{i, j=1, \ldots, d}$ and $b = (b_1, \ldots, b_d)$, defined in $I$ and $I\times\Rd$, respectively, where $I$ is an open right halfline and $s\in  I$.
Throughout the paper we assume that the coefficients $q_{ij}$ are bounded and that
the operators $\mathcal{A}(t)$ are uniformly elliptic, i.e., there exists a positive constant $\eta_0$ such that
\begin{equation}
\label{elliptic}
\langle Q(t)\xi,\xi\rangle\geq\eta_0|\xi|^2,\qquad\;\, \xi\in \R^d,\;\, t\in I.
\end{equation}
Problem \eqref{p_e} arises (after time reversal) as a Kolmogorov equation of the stochastic differential equation
\begin{eqnarray*}
dX(t,s, x) = b(t, X(t,s,x))dt + \sigma(t)dW_t, \quad X(s, s,x)=x,
\end{eqnarray*}
where $W_t$ is a standard $d$-dimensional Brownian motion, and $Q(t) = \sigma(t) (\sigma(t))^*/2$. There is no need that $b$ be bounded to have existence in the large of a strong solution for every $x\in \R^d$ and to define the transition evolution operator $f\mapsto {\mathbb E}(f(X(t,s,x)))$ that leads to \eqref{p_e}, see e.g., \cite{Cerrai, gikhman}.

It is well known that the usual $L^p$ spaces with respect  to the Lebesgue measure $dx$ are not a natural setting for elliptic and parabolic operators with unbounded coefficients, unless quite strong growth assumptions are imposed on their coefficients.
For instance, if $\varepsilon >0$  no realization of the operator  $(\mathcal{A}\zeta)(x)=\zeta''(x)\pm {\rm sign}(x)|x|^{1+\varepsilon}\zeta'(x)$  in   $L^p(\R, dx)$  generates a strongly continuous semigroup,  as it has been shown in \cite{abdel}.
Much better settings are $L^p$ spaces with respect to the so called evolution systems of  measures $\{\mu_t: t \in I\}$.
An evolution system of measures $\{\mu_t: t \in I\}$ for a Markov evolution operator $G(t,s)$ is a family of Borel probability measures in $\Rd$ satisfying
\begin{equation}
\int_{\Rd}G(t,s)f \mu_t(dx)= \int_{\Rd} f \mu_s(dx), \qquad\;\, t>s \in I,\;\,f\in C_b(\Rd).
\label{invariant-0}
\end{equation}

As noticed e.g., in \cite{DPR}, the family $\{\mu_t: t\in I\}$ is the natural nonautonomous counterpart of the invariant measure for a Markov semigroup in the autonomous case.
If an evolution system of measures exists, formula \eqref{invariant-0} allows easily to prove that for $p\geq 1$ we have
\begin{equation}
\label{contr}
\|G(t,s)f \|_{L^p(\R^d, \mu_t)} \leq \| f \|_{L^p(\R^d, \mu_s)}, \qquad\;\, t\geq s, \;f\in C_b(\R^d),
\end{equation}
and consequently  $G(t,s)$ may be extended to a contraction  (still denoted by   $G(t,s)$) from $L^p(\Rd,\mu_s)$ into $L^p(\Rd,\mu_t)$ for any $t>s$.
However, in general  the spaces $L^p(\Rd,\mu_s)$ and $L^p(\Rd,\mu_t)$ are different if $t\neq s$, and  the classical theory of evolution operators in fixed Banach spaces cannot be used.

Under mild assumptions on $Q$ and $b$, in \cite{KunLorLun09Non}  the existence of a Markov evolution operator $G(t,s)$ associated to the family  $\{\mathcal{A}(t)\}_{t\in I}$, as well as the existence of a tight evolution system of measures $\{\mu_t:\;t\in I\}$, have been proved.
In this paper we study  the asymptotic behavior of  $G(t,s)$ as $t\to +\infty$, and we prove hypercontractivity results in the spaces $L^p(\R^d, \mu_t)$.

In addition to the basic hypotheses of \cite{KunLorLun09Non}, we assume that the quadratic form associated to the Jacobian matrix $\nabla_xb$ is uniformly negative definite, namely there exists $r_0<0$ such that
\begin{eqnarray*}
\langle \nabla_xb(t,x)\xi, \xi \rangle \leq r_0|\xi|^2, \qquad\;\, t\in I, \;\,x,\xi \in \R^d.
\end{eqnarray*}
This is a strong assumption that yields strong qualitative results, such as the pointwise gradient estimate
\begin{equation}
\label{p=1}
|(\nabla_x G(t,s)f)(x)|\leq e^{r_0(t-s)}(G(t,s)|\nabla f|)(x),
\end{equation}
valid  for every $f \in C^1_b(\Rd)$, $t\geq s$, $x \in \Rd$, and its consequence
\begin{equation}
\label{comp_as}
\Vert\,|\nabla_x G(t,s)f|\,\Vert_{L^p(\Rd,\mu_t)}\leq c_pe^{r_0(t-s)} \Vert f\Vert_{L^p(\Rd,\mu_s)},\qquad\;\, t\geq s+1, \;\,f\in L^p(\Rd,\mu_s),
\end{equation}
 see  \cite{KunLorLun09Non}.
The starting point of our analysis is the proof of
the logarithmic Sobolev inequality (in short LSI) for the  measures $ \mu_t $, in the form
\begin{align}
\int_{\Rd} |f|^p\log |f|\,\mu_t(dx)\le &\frac{1}{p}\left (\int_{\Rd}|f|^p\mu_t(dx)\right )\log\left (\int_{\Rd}|f|^p\mu_t(dx)\right )\notag\\
&+ pC\int_{\Rd}|f|^{p-2}|\nabla f|^2 \chi_{\{f\neq 0\}}\mu_t(dx),
\label{log_intro}
\end{align}
for any $t\in I$, any $p\in (1,+\infty)$ and some positive constant $C$, independent of $f\in C^1_b(\R^d)$, $t$ and  $p$. The  gradient estimate \eqref{p=1} allows us to follow the method used by Deuschel and Stroock \cite{DS}  in the autonomous case for the invariant measure $\mu$ of a Markov semigroup, but the proof is   much more complicated because the measures $\mu_t$ depend explicitly on time. In particular, we have to deal with the regularity of $\mu_t$ with respect to $t$. We use in a crucial way a differentiability property, 
\begin{equation}\label{thanks}
\frac{d}{dt}\int_{\Rd}f(x) \mu_t(dx)= -\int_{\Rd} (\mathcal{A}(t)f)(x)
\mu_t(dx),\qquad\;\, t \in I,
\end{equation}
valid for every $f\in C^2_b(\Rd)$, constant outside a compact set.

Under Hypotheses \ref{hyp1}, the operator $G(t,s)$ is bounded from $L^q(\Rd,\mu_s)$ into $W^{1,q}(\Rd,\mu_t)$ for    $I\ni s<t$,  $q\in (1,+\infty)$  (\cite{KunLorLun09Non}).
The question whether it is bounded (or, even better, contractive) from $L^q(\Rd,\mu_s)$ into $L^r(\Rd,\mu_t)$ for some  $r>q$,  is particularly meaningful. Indeed,
Sobolev embedding theorems do not hold in  general, as the simple example of the standard Gaussian measure in $\R$ shows, hence it  is not obvious that $G(t,s)$ improves summability.

We prove that in fact this is the case. We follow the method of  \cite{GeisLun09} that deals with time depending Ornstein-Uhlenbeck operators,
\begin{equation}\label{Ou_intro}
(\mathcal{A}(t)\zeta)(x)= \frac{1}{2}\textrm{Tr}(Q(t)(Q(t))^*D^2\zeta(x))
+\langle B(t)x+f(t), \nabla\zeta(x) \rangle,\qquad\;\,t\in\R,\;\,x \in \Rd,
\end{equation}
and that is, in its turn,  an extension of the method of Gross (\cite{Gro75Log}) to a nonautonomous setting, where the LSI \eqref{log_intro} plays a fundamental role. However,  in \cite{GeisLun09}   there are an explicit representation formula for the evolution operator and explicit representation formulae for the measures $\mu_t$, that are used in the proof of the LSI and  of the hypercontractivity. On the contrary, in our case $G(t,s)$ and $\mu_t$ are not explicit.

Another important consequence of \eqref{log_intro} is the Poincar\'e
inequality
\begin{equation}
\|f-m_s(f)\|_{L^p(\Rd,\mu_s)}\le C_p\|\,|\nabla f|\,\|_{L^p(\Rd,\mu_s)},\qquad\;\,f\in W^{1,p}(\Rd,\mu_s),\;\,s\in I,
\label{Poi-intro}
\end{equation}
where $m_s(f) = \int_{\R^d} f d\mu_s$, and $C_p$ is a positive constant, independent
of $f$ and $s$.
First   \eqref{Poi-intro} is proved for $p=2$, then, by a bootstrap argument, we extend it to   $p>2$.

Using the Poincar\'e inequality with $p=2$ and the hypercontractivity of   $G(t,s)$, we compare the asymptotic behavior (as $t\to +\infty$) of  $\Vert G(t,s)f -m_sf \Vert_{ L^p(\Rd,\mu_t)}$ and $\Vert\,|\nabla_x G(t,s)f |\,\Vert_{L^p(\Rd,\mu_t)}$. Precisely, we prove the   equality $\mathfrak{A}_p=\mathfrak{B}_p$, for any $p\in (1,+\infty)$, where
\begin{align*}
\mathfrak{A}_p=&\big\{\omega\in\R: \exists M_{p,\omega}>0~{\rm s.t.}~\|G(t,s)f-m_s(f)\|_{L^p(\R^d,\mu_t)}\le M_{p,\omega}e^{\omega (t-s)}\|f\|_{L^p(\R^d,\mu_s)},\\
&~~~\;I\ni s<t,\;\,f\in L^p(\R^d,\mu_s)\big\};\\[1mm]
\mathfrak{B}_p=&\big\{\omega\in\R: \exists N_{p,\omega}>0~{\rm s.t.}~\|\,|\nabla_x G(t,s)f|\,\|_{L^p(\R^d,\mu_t)}\le N_{p,\omega}e^{\omega ( t-s)}\|f\|_{L^p(\R^d,\mu_s)},\\
&~~~\;s,t\in I,\, t-s\ge 1,\;\,f\in L^p(\R^d,\mu_s) \big\}.
\end{align*}
We also show that $\mathfrak{A}_p$ is independent of $p$. Then, estimate \eqref{comp_as} implies that $r_0\in \mathfrak{B}_p$, and therefore  $\Vert G(t,s)-m_s\Vert_{{\mathcal L}(L^p(\Rd,\mu_s), L^p(\Rd,\mu_t))}$ decays exponentially to zero, as $t\to +\infty$.

In the case of the nonautonomous Ornstein-Uhlenbeck operators \eqref{Ou_intro} we prove the conjecture in \cite{GeisLun09} on the optimal decay estimate of   $\|G(t,s)f-m_s(f)\|_{L^2(\R^d,\mu_t)}$ as $t\to +\infty$ and we show that the same optimal decay estimate holds also replacing
$L^2(\Rd,\mu_t)$  by $L^p(\Rd,\mu_t)$ for any $p>1$. See Subsection \ref{sect-appl}.

The   equality $\mathfrak{A}_p=\mathfrak{B}_p$ was already   proved in \cite{LorLunZam10} in the case that  the coefficients $q_{ij}$, $b_i$ are
periodic with respect to   $t$, under more restrictive assumptions   and only for $p\ge 2$.

Since most of our asymptotic behaviour results are expressed in terms of the measures $\mu_t$, the asymptotic behaviour of $\mu_t$ as $t\to +\infty$ is also of interest. The explicit determination of all the weak$^*$ limit measures of $\mu_t$ as $t\to +\infty$ is out of hope in general. Here, we consider the case where the coefficients $q_{ij}$ and $b_j$ ($i,j=1,\ldots,d$) converge as $t\to +\infty$, and we prove that $\mu_t$ weakly$^*$ converges to the invariant measure $\mu$ of the semigroup generated by the limiting operator.

Differently from \cite{GeisLun09} and \cite{LorLunZam10} all the results of this paper are proved without using the evolution semigroup associated to the evolution family $G(t,s)$.

Our results heavily rely on the LSI \eqref{log_intro} which is proved using the
pointwise gradient estimate \eqref{p=1}. Even  in the autonomous case \eqref{p=1} does not hold when the diffusion coefficients depend on $x$ and they do not satisfy the condition in \cite{wang}. This is the reason why we consider  diffusion coefficients depending only on $t$.

The paper is organized as follows. In Section \ref{sect-2} we introduce our hypotheses and we collect some  preliminary results.
Section \ref{sec_log_sob} is devoted to establish the logarithmic Sobolev inequality and, as a consequence,
the compactness of the embedding $W^{1,p}(\Rd,\mu_s)\hookrightarrow L^p(\Rd,\mu_s)$, for any $p\ge 2$ and any $s\in I$,
and the compactness of the operator $G(t,s)$ from $L^p(\Rd,\mu_s)$ into $L^p(\Rd,\mu_t)$ for any $t>s$ and any $p\in (1,+\infty)$.
Next, in Section \ref{sec_hyper} we prove the hypercontractivity of $G(t,s)$.
In Section \ref{sec_poincare} we establish the Poincar\'e inequality for every $p\ge 2$, we prove the exponential convergence of $G(t,s)$ to $m_s$ in  ${\mathcal L}(L^p(\Rd,\mu_s), L^p(\Rd,\mu_t))$ and we characterize it in terms of the convergence rate to zero of  $\nabla_x G(t,s) $.
In Section \ref{sect-asympt-meas} we study the asymptotic behaviour of $\mu_t$ when the coefficients converge as $t\to +\infty$.
Finally, in Section \ref{examples} we briefly comment on our hypotheses and give examples of nonautonomous operators to which the results of this paper may be applied.

\subsection*{Notations}
Let $\Omega$ be an open set or the closure of an open set in $\R^N$, and let $k\in \N \cup \{+ \infty\}$. We consider the usual spaces $C(\Omega)$ and $C^k(\Omega)$, as well as $C^k_b(\Omega)$, the subspace of $C^k(\Omega)$ consisting of
bounded functions with bounded derivatives up to the $k$-th order. For $\alpha\in (0,1)$, $C^{ \alpha}(\Omega)$ is the usual H\"older space; we use the subscript ``loc'' to denote the space of all $f\in C(\Omega)$ which are $\alpha$-H\"older continuous in any compact subset of $\Omega$.
We use the subscript ``$c$'' (resp. ``$0$'') instead of ``$b$''  for the subsets of the above spaces consisting of functions  with compact support (resp.  vanishing at infinity).

If $J \subset \R$ is an interval,  the parabolic H\"older spaces $C^{\alpha/2,\alpha}(J \times\Rd)$ ($\alpha\in (0,1)$) and $C^{1,2}(J \times\Rd)$ are defined in the usual way;
 the subscript ``loc'' has the same meaning as above.

About partial derivatives, the notations $D_tf:=\frac{\partial f}{\partial t}$,
$D_if:=\frac{\partial f}{\partial x_i}$, $D_{ij}f:=\frac{\partial^2f}{\partial x_i\partial x_j}$ are extensively used.

About matrices and vectors, we denote by $\textrm{det}(Q)$, $\textrm{Tr}(Q)$ and $\langle x,y\rangle$ the determinant, the trace of the square matrix
$Q$ and the  scalar product
of the vectors $x,y\in\Rd$, respectively. The adjoint  of $Q$ is denoted by $Q^*$.

By $\chi_A$ and $\one$ we denote, respectively, the characteristic function of the set $A\subset\Rd$ and the function which is identically
equal to 1 in $\Rd$. The  ball in $\R^d$ centered at $0$ with  radius $r>0$ is denoted by $B(0,r)$. The Lebesgue measure in $\R^d$ is denoted by $dx$.

\section{Assumptions and preliminary results}
\label{sect-2}

Let $I$ be an open right  halfline. For $t\in I$ we consider
linear second order differential
operators $ \mathcal{A}(t) $
defined on smooth functions $\zeta$ by
\begin{align*}
(\mathcal{A}(t)\zeta)(x)&=\sum_{i,j=1}^d q_{ij}(t)D_{ij}\zeta(x)+
\sum_{i=1}^d b_i(t,x)D_i\zeta(x)\\
&= \textrm{Tr}(Q(t)D^2\zeta(x))+ \langle b(t,x), \nabla \zeta(x)\rangle,\qquad\;\, x\in\R^d,
\end{align*}
under the following assumptions on their coefficients.

\begin{hyp}\label{hyp1}
\begin{enumerate}[\rm (i)]
\item
$q_{ij}\in C^{\alpha/2}_{\rm loc}(I)$ and
$b_{i}\in C^{\alpha/2,\alpha}_{\rm loc}(I\times \R^d)$ $(i,j=1,\dots,d)$ for
some $\alpha \in (0,1)$;
\item
for every $t\in I$, the matrix $Q(t)=[q_{ij}(t)]_{i,j=1, \ldots, d}$
is symmetric and there exist $0<\eta_0<\Lambda$ such that
\begin{equation}
\label{ell}
\eta_0|\xi|^2\le \langle Q(t)\xi,\xi\rangle\leq \Lambda|\xi|^2,\qquad\;\, (t,\xi )\in I\times \Rd;
\end{equation}
\item
there exists   $\varphi\in C^2(\R^d)$ with positive values such that
\begin{equation}
\label{Lyapunov}
\;\;\;\;\;\qquad\lim_{|x|\to +\infty}\varphi(x)=+\infty \quad\textrm{and}\quad
(\mathcal{A}(t)\varphi)(x)\leq a-c\,\varphi(x),  \quad (t,x)\in I\times \Rd
\end{equation}
for some positive constants $a$ and $c$;
\item
the first order spatial derivatives of $b_i$ exist, belong to
$C^{\alpha/2,\alpha}_{\rm loc}(I\times \Rd)$ for any $i=1, \dots,d$, and
there exists  $r_0<0$ such that
\begin{equation}
\label{b}
\langle \nabla_x b(t,x)\xi,\xi\rangle\leq r_0|\xi|^2,\qquad\;\, (t,x)\in I\times \Rd,\;\, \xi \in \Rd.
\end{equation}
\end{enumerate}
\end{hyp}

Throughout the paper we assume that all the conditions in Hypotheses \ref{hyp1} are satisfied, if not otherwise specified.
Let us make some comments.

\begin{rmk}
{\rm As easily seen, condition \eqref{b} is equivalent to
\begin{equation}
\label{monotonia}
\langle b(t,x)- b(t,y),x-y\rangle\le  r_0|x-y|^2 , \qquad\;\, t\in I, \;\,x,y\in \R^d.
\end{equation}
Then:}
\begin{itemize}
\item [(a)] {\rm taking $y=0$ we get $\langle b(t,x),x \rangle\le  \langle b(t,0),x \rangle
+r_0|x |^2$, and since $r_0<0$, for
any $[a,b]\subset I$ there exists a positive constant $C_{a,b}$ such that}
\begin{equation}
\langle b(t,x),x\rangle\le C_{a,b},\qquad\;\,t\in [a,b],\;\,x\in\Rd.
\label{inequality-b}
\end{equation}
{\rm This estimate will be used later, in the proof of the LSI inequality and of the hypercontractivity.}
\item[(b)] {\rm If  $b(\cdot , \overline{y})$ is bounded in $I$ for some $\overline{y}\in \R^d$, the function
$\varphi(x):= e^{\delta |x-\overline{y}|^2}$ satisfies Hypothesis \ref{hyp1}(iii)
 if $\delta >0$ is small enough. Similarly, if $\langle b(t, x), x\rangle \leq -C|x|^{\beta}$ for $|x|$ large,
with $C>0$, $\beta >1$ independent of $t$ and $x$, then the function $\varphi(x):= e^{\delta |x |^{\beta}}$ satisfies Hypothesis \ref{hyp1}(iii)
if $\delta >0$ is small enough. See Section \ref{examples} for more details. }
 \end{itemize}
\end{rmk}

Under Hypotheses \ref{hyp1}(i)-(iii), in \cite{KunLorLun09Non} a Markov evolution operator $G(t,s)$   associated to \eqref{p_e} has been constructed. Here we recall its main properties.

For every continuous and bounded function $f:\Rd\to\R$ and for any $s \in I$, the function $(t,x)\mapsto (G(t,s)f)(x)$
is the unique bounded classical solution to the Cauchy problem
\begin{eqnarray*}
\left\{
\begin{array}{lll}
D_tu(t,x)=\mathcal{A}(t)u(t,x), &\quad t>s,  &x \in \Rd,\\[1mm]
u(s,x)=f(x),&& x \in \Rd.
\end{array}
\right.
\end{eqnarray*}
Then,  $G(\cdot,s)f\in C_b([s,+\infty)\times \Rd)\cap C^{1,2}((s,+\infty)\times\Rd)$.
Moreover,
\begin{equation}\label{rep_nucleo}
(G(t,s)f)(x)=\int_{\R^d}g(t,s,x,y)f(y)dy,\qquad\;\,s<t,\;\,x\in\R^d,
\end{equation}
where $g:\{(t,s)\in I\times I: t> s\}\times\R^d\times\R^d\to\R$ is a positive function
such that $\|g(t,s,x,\cdot)\|_{L^1(\Rd)}=1$ for any $t,s\in I$, with $t>s$, and any $x\in\Rd$
(\cite[Prop. 2.4]{KunLorLun09Non}).
%
%

By \cite[Thm. 5.4]{KunLorLun09Non}
there exists an evolution system of measures $\{\mu_t: t\in I\}$ for $G(t,s)$. The Lyapunov function $\varphi$ is in $L^1(\R^d, \mu_t)$ for every $t\in I$, and there exists a constant $M>0$ such that
\begin{equation}
\int_{\Rd} \varphi(y)\mu_t(dy)\leq M, \qquad\;\, t \in I.
\label{const-M}
\end{equation}
This implies that the family of measures $\{\mu_t: t\in I\}$ is {\emph{tight}},
that is for every  $\varepsilon >0$ there exists $R=R_{\varepsilon}>0$ such that
$\mu_t(\Rd\setminus B(0, R))\leq \varepsilon$ for any $t\in I$.

Moreover, \eqref{rep_nucleo} implies $|(G(t,s)f)(x)|^p\leq (  G(t,s)|f|^p)(x)$ for every $f \in C_b(\Rd)$, $t>s\in I$, $x\in \R^d$ and $p\geq 1$. Integrating with respect to $\mu_t$ and using
\eqref{invariant-0} we obtain
\begin{equation}
\label{contraz}
\|G(t,s)f\|_{L^p(\R^d, \mu_t)}  \leq \| f\|_{L^p(\R^d, \mu_s)} , \qquad\;\, t>s\in I,
\end{equation}
and since $C_b(\Rd)$ is dense in $L^p(\R^d, \mu_s)$, $G(t,s)$ may be extended to
a contraction (still denoted by $G(t,s)$)  from  $L^p(\R^d, \mu_s)$ to $L^p(\R^d, \mu_t)$, such that  \eqref{invariant-0} holds for every  $f\in L^p(\R^d, \mu_s)$.

If also Hypothesis \ref{hyp1}(iv) holds, then $\{\mu_t:\;t\in I\}$ is the {\em unique} tight evolution system of measures. See next Remark \ref{rem:tight}. Moreover, in this case the following results have been proved in  \cite[Thm. 4.5]{KunLorLun09Non} and \cite[(proof of) Prop. 3.3]{LorZam09}.

\begin{prop}
\begin{enumerate}[\rm (i)]
\item
The pointwise gradient estimate
\begin{equation}\label{grad_est_punt}
|(\nabla_x G(t,s)f)(x)|^p\leq e^{pr_0(t-s)}(G(t,s)|\nabla f|^p)(x),
\end{equation}
holds for every $f \in C^1_b(\Rd)$, $t\geq s$, $x \in \Rd$ and $p\in [1,+\infty)$.
\item
For each $p\in (1,+\infty)$ and  $f\in L^p(\R^d,\mu_s)$, the function $G(t,s)f$ belongs to $W^{1,p}(\R^d,\mu_t)$ and
there exists a   constant $c_p$, independent of $f$, such that
\begin{equation}\label{grad_est_norm_p}
\;\;\;\;\;\;\;\;\;\,\Vert\,|\nabla_x G(t,s)f|\,\Vert_{L^p(\Rd,\mu_t)}\leq c_p(t-s)^{-1/2}\Vert f\Vert_{L^p(\Rd,\mu_s)},\qquad\;\, s< t \le s+1.
\end{equation}
\end{enumerate}
\end{prop}

\begin{lemm}\label{density_mus}
Under Hypotheses $\ref{hyp1}(i)$-$(iii)$, for every $s\in I$ the
measure $\mu_s$ is absolutely continuous with
respect to the Lebesgue measure. More precisely, $\mu_s=\rho(s,\cdot)dx$ for  some strictly
positive and locally H\"older continuous function $\rho:I\times\Rd\to \R$.
\end{lemm}

\begin{proof}  The measures $\mu_s$ are absolutely continuous with respect to the Lebesgue measure by \cite[Prop. 5.2]{KunLorLun09Non}. Local H\"older continuity  and positivity of $\rho$ is a consequence of  \cite[Sect. 3]{BogKryRoc01}. More precisely,
by  \cite[Thm. 3.8]{BogKryRoc01} we know that the measure
$\nu$ on $I\times \R^{d}$,  defined on products of Borel sets $A\subset I$ and $B\subset\Rd$  by
\begin{eqnarray*}
\nu(A\times B)=\int_A\mu_s(B)ds,
\end{eqnarray*}
has a positive density $\rho$ with respect to the Lebesgue measure, and  $\rho\in C^{\gamma}_{\rm loc}(I\times\R^d)$ for each $\gamma\in (0,1)$.
Hence, for each  Borel set $A\subset I$ and for   $\zeta \in C^{\infty}_c(\Rd)$ we have
\begin{eqnarray*}
\int_Ads\int_{\Rd}\zeta \mu_s(dx)=\int_{I\times\R^d}\chi_A\zeta
d\nu=\int_A ds\int_{\Rd}\zeta\rho(s,\cdot)dx.
\end{eqnarray*}
Since  $A$ is arbitrary,
\begin{equation}
\int_{\Rd}\zeta(x) \mu_s(dx)=\int_{\Rd}\zeta(x)\rho(s,x)dx,\qquad\;\,{\rm
for~a.e.}~s\in I.
\label{uguali}
\end{equation}
Let us prove that \eqref{uguali}   in fact holds for every $s\in I$, showing that both sides are continuous with respect to $s$. The right hand side is  continuous since $\rho$ is. By \eqref{invariant-0} the left hand side is equal to $\int_{\R^d}G(r,s)\zeta\mu_{r}(dx)$ for any $r>s$, and   the function $s\mapsto
G(r,s)\zeta$ is continuous  in $I \cap (-\infty, r]$, by \cite[Lemma 3.2]{KunLorLun09Non}. Then,
  $s\mapsto \int_{\Rd}\zeta\mu_s(dx)$ is continuous in $I \cap (-\infty, r]$, and since $r$ is arbitrary, it is continuous in $I$.   Then,  \eqref{uguali}  holds for each $s\in I$.

Let $B\subset \R^d$ be any Borel set.  Then $\chi_B$ is the a.e. limit (with respect to the
Lebesgue measure and, hence, with respect to each $\mu_s$)
of a bounded sequence of smooth
and compactly supported functions. From \eqref{uguali} we infer
\begin{equation}
\int_B\mu_s(dx)=\int_B\rho(s,x)dx,\qquad\;\,s\in I,
\label{uguali-1}
\end{equation}
and the proof is complete.
\end{proof}

The following lemma will be frequently used in the next sections.
Its (easy) proof follows from a standard truncation argument and the  equivalence of the Sobolev spaces $W^{1,p}(B(0,R), \rho dx)$ and $W^{1,p}(B(0,R),  dx)$ for every $R>0$,  if $\rho$ is a locally bounded function with positive infimum on every ball.

\begin{lemm}\label{density}
Let $\mu (dx) = \rho (x)dx$ be a probability measure on $\R^d$, where $\rho: \R^d\to \R$ is a locally bounded function with positive infimum on every ball. Then $C^{\infty}_c(\Rd)$ is dense in $W^{1,p}(\Rd,\mu)$.

In particular, under Hypotheses $\ref{hyp1}(i)$-$(iii)$, $C^{\infty}_c(\Rd)$ is dense in $W^{1,p}(\Rd,\mu_s)$ for each  $s\in I$ and   $p\in [1,+\infty)$.
\end{lemm}

As a first consequence, we obtain a decay estimate for the gradient of $G(t,s)f$ as $t\to +\infty$.

\begin{prop}\label{cor:decay}
For every $p\geq 1$ there is $K_p>0$ such that for each $s\in I$, $t\geq s+1$ we have
\begin{equation}
\label{contraz_grad0}\|\,|\nabla_xG(t,s)f|\,\|_{L^p(\R^d, \mu_t)}    \leq K_p e^{r_0(t-s)}\| f \|_{L^p(\R^d, \mu_s)}, \qquad\;\, f\in   L^p(\Rd,\mu_s).
\end{equation}
\end{prop}
\begin{proof}
Integrating \eqref{grad_est_punt} with respect to $\mu_t$ and   using
\eqref{invariant-0} we obtain
\begin{equation}
\label{contraz_grad}
\|\,|\nabla_xG(t,s)f|\,\|_{L^p(\R^d, \mu_t)}    \leq e^{r_0(t-s)} \| \,|\nabla f|\, \|_{L^p(\R^d, \mu_s)} , \qquad\;\, t \geq s\in I,
\end{equation}
for each $f \in C^1_b(\Rd)$, and hence for each $f\in W^{1,p}(\R^d,\mu_s)$ by Lemma \ref{density}.
If $t\geq s+1$ and $f\in L^p(\R^d, \mu_s)$,
\begin{eqnarray*}
\|\,|\nabla_xG(t,s)f|\,\|_{L^p(\R^d, \mu_t)}  = \|\,|\nabla_xG(t,s+1)G(s+1,s)f|\,\|_{L^p(\R^d, \mu_t)},
\end{eqnarray*}
and the statement follows from \eqref{contraz_grad} and \eqref{grad_est_norm_p}. \end{proof}

For every $t \in I$ and $f \in L^1(\Rd, \mu_t)$ we denote by $m_t(f)$ the average of $f$ with respect to $\mu_t$, i.e.,
\begin{equation}
\label{media}
m_t(f)= \int_{\Rd} f(x) \mu_t(dx).
\end{equation}

In the following lemma we prove that $G(t,s)f$ converges to $m_s(f)$ as $t\to +\infty$. It is a first step towards better asymptotic behavior results, and will be used in the proof of the LSI inequality.
The same result has been proved in \cite{LorLunZam10} in the case of time periodic coefficients; here estimate \eqref{grad_est_punt} allows us to give a much simpler proof.

\begin{lemm}\label{norma_p}
For every  $s \in I$ and $p\in [1, +\infty)$ we have
\begin{eqnarray*}
\lim_{t \to +\infty}\Vert G(t,s)f-m_s(f)\Vert_{L^p(\Rd,\mu_t)}=0,\qquad\;\, f \in L^p(\Rd,\mu_s).
\end{eqnarray*}
\end{lemm}
\begin{proof}
Let   $f\in C^\infty_c(\Rd)$. Then
\begin{eqnarray*}
(G(t,s)f)(x)-m_s(f)= \int_{\Rd} ((G(t,s)f)(x)-(G(t,s)f)(y))\mu_t(dy), \quad t\geq s,\;\, x\in \Rd.
\end{eqnarray*}
Set $B_t:= B(0,e^{-r_0 t/2})$, where $r_0$ is defined in Hypothesis \ref{hyp1}(iv), and $A_t:= \Rd\setminus B_t$.
For  $t\geq s$ and $x \in \Rd$ we have
\begin{align}\label{faccina}
\left|(G(t,s)f)(x)-m_s(f)\right|\le &\int_{A_t} |(G(t,s)f)(x)-(G(t,s)f)(y)|\mu_t(dy)\nonumber\\
&+ \int_{B_t} |(G(t,s)f)(x)-(G(t,s)f)(y)|\mu_t(dy)\nonumber\\
\leq & 2 \Vert f\Vert_\infty \mu_t(A_t)+ \Vert\,|\nabla_x G(t,s)f|\,\Vert_\infty \int_{B_t} |x-y|\mu_t(dy)\nonumber\\
\leq & 2 \Vert f\Vert_\infty \mu_t(A_t)+  e^{r_0(t-s)}\Vert\,|\nabla f|\,\Vert_\infty \left (|x|+ \int_{B_t}|y|\mu_t(dy)\right )\nonumber\\
\leq & 2\Vert f\Vert_\infty \mu_t(A_t)+\Vert\,|\nabla f|\,\Vert_\infty\left(e^{r_0(t-s)}|x|+e^{-r_0 s}e^{\frac{1}{2}r_0 t}\right),
\end{align}
where we have used \eqref{grad_est_punt}. It   follows that
\begin{align*}
\Vert G(t,s)f-m_s(f)&\Vert_{L^p(\Rd,\mu_t)}^p\leq \int_{A_t} \left|(G(t,s)f)(x)-m_s(f)\right|^p\mu_t(dx)\notag\\
&\quad\quad\quad\quad\quad\quad+ \int_{B_t} \left|(G(t,s)f)(x)-m_s(f)\right|^p\mu_t(dx)\notag\\
& \leq 2^p \Vert f\Vert_\infty^p \mu_t(A_t)+2^{p-1}\Big(2 \Vert f\Vert_\infty\mu_t(A_t)+\Vert\,|\nabla f|\,\Vert_\infty e^{-r_0 s}e^{\frac{1}{2}r_0 t}\Big)^p\notag\\
& \quad\quad + 2^{p-1}( e^{r_0(t-s)}\Vert\,|\nabla f|\,\Vert_\infty)^p\int_{B_t}|x|^p \mu_t(dx)\notag\\
& \leq 2^{p} \Vert f\Vert_\infty^p \mu_t(A_t)\Big(1+ 2^{2p-2} (\mu_t(A_t))^{p-1}\Big )\notag\\
&\quad\quad + 2^{p-1}\Vert\,|\nabla f|\,\Vert_\infty^p e^{-pr_0 s}e^{\frac{1}{2}pr_0 t}(2^{p-1}+1).
\end{align*}
We recall that the family of measures $\{\mu_t: t\in I\}$ is tight. Therefore, since
the radius of the ball $B_t$ tends to $+\infty$ as $t\to +\infty$ and $A_t=\Rd\setminus B_t$,
$\mu_t(A_t)$ tends to $0$ as $t\to +\infty$. This shows that $\Vert G(t,s)f-m_s(f)\Vert_{L^p(\Rd,\mu_t)}$ vanishes as $t\to +\infty$.

Since $C^\infty_c(\Rd)$ is dense in $L^p(\Rd,\mu_s)$ the statement follows.
\end{proof}

\begin{rmk}
\label{rem:tight}
{\rm  In the proof of the previous lemma the only property of the set of probability measures $\{\mu_s:\;s\in I\}$ that we use is the tightness.  In particular, by
  \eqref{faccina} for every tight  evolution system of measures
$\{\nu_s:\;s\in I\}$ and for every $f\in C^{\infty}_c(\R^d)$, the mean values of $f$ with respect to   $\mu_s$ and to $\nu_s$ are the pointwise limit of $G(t,s)f$ as $t\to +\infty$, so that they coincide for every $s$. Then, $\mu_s=\nu_s$, i.e., $\{\mu_s:\;s\in I\}$
is the {\em unique} tight evolution system of measures for $G(t,s)$.}
\end{rmk}


\section{Logarithmic Sobolev inequality}\label{sec_log_sob}


Throughout this section we set $0\log 0 =0$. First of all, we prove a crucial preliminary lemma.

\begin{lemm}\label{derivative}
Assume that Hypotheses $\ref{hyp1}(i)$-$(iii)$ hold. Then:
\begin{enumerate}[\rm (i)]
\item
if $f\in C^2_b(\Rd)$ is constant outside a compact set $K\subset \Rd$, then the function $r \mapsto \int_{\Rd}f(x)\mu_r(dx)$ is continuously differentiable in $I$ and
\begin{eqnarray*}
\frac{d}{dr}\int_{\Rd}f(x) \mu_r(dx)= -\int_{\Rd} (\mathcal{A}(r)f)(x)
\mu_r(dx),\qquad\;\, r \in I.
\end{eqnarray*}
\item
Let $[a,b]\subset I$. If $f\in C^{1,2}_b([a,b]\times\Rd)$ and  $f(r,\cdot)$ is constant outside a compact set $K$ for every $r\in [a,b]$, then the function $r \mapsto \int_{\Rd}f(r,x)\mu_r(dx)$ is continuously differentiable in $[a,b]$ and
\begin{eqnarray*}
\;\;\;\;\;\;\;\;\;\;\;\frac{d}{dr}\int_{\Rd}f(r,x) \mu_r(dx)=\int_{\Rd}D_rf(r,x)\mu_r(dx) -\int_{\Rd} (\mathcal{A}(r)f(r, \cdot))(x)\mu_r(dx),
\end{eqnarray*}
for every  $r \in [a,b]$.
\end{enumerate}
\end{lemm}
\begin{proof}
(i) To begin with, let us observe that, for any $t\in I$,
the function $G(t,\cdot)\mathcal{A}(\cdot)f$ is continuous and bounded in $(I\cap(-\infty,t])\times\Rd=:I_t\times\Rd$.
Indeed, for any $\sigma, \sigma_0 \in I_t$,
\begin{align}
|(G(t,\sigma)\mathcal{A}(\sigma)f)(x)-(G(t,\sigma_0)&\mathcal{A}(\sigma_0)f)(x_0)|\notag\\
&\leq |(G(t,\sigma)(\mathcal{A}(\sigma)f-\mathcal{A}(\sigma_0)f))(x)|\notag\\
&\quad+|((G(t,\sigma)-G(t,\sigma_0))\mathcal{A}(\sigma_0)f)(x)|\notag\\
&\quad+ |(G(t,\sigma_0)\mathcal{A}(\sigma_0)f)(x)-(G(t,\sigma_0)\mathcal{A}(\sigma_0)f)(x_0)|\notag\\
&\leq \Vert \mathcal{A}(\sigma)f-\mathcal{A}(\sigma_0)f  \Vert_\infty\notag\\
&\quad+ \Vert(G(t,\sigma)-G(t,\sigma_0))\mathcal{A}(\sigma_0)f\Vert_\infty\notag\\
&\quad+ |(G(t,\sigma_0)\mathcal{A}(\sigma_0)f)(x)-(G(t,\sigma_0)\mathcal{A}(\sigma_0)f)(x_0)|.
\label{star}
\end{align}
Clearly, the first and the third addenda in the right-hand side of \eqref{star} vanish as $\sigma\to\sigma_0$ and
$x\to x_0$, respectively. Concerning the second one, we observe that it tends to $0$ as $\sigma
\to \sigma_0$ since the function $G(t,\cdot)g$ is continuous in $I_t$ with values in $C_b(\Rd)$ for any $g \in C_0(\Rd)$ by \cite[Lemma 3.2]{KunLorLun09Non},
and $\mathcal{A}(\sigma_0)f \in C_c(\Rd)$.
Again by \cite[Lemma 3.2]{KunLorLun09Non}, for $I\ni r, \,r+h <t$   we have
\begin{eqnarray*}
(G(t,r+h)f)(x)-(G(t,r)f)(x)= -\int_{r}^{r+h}(G(t,\sigma)\mathcal{A}(\sigma)f)(x)d\sigma.
\end{eqnarray*}
Integrating  over $\R^d$  with respect to $\mu_t$, we get
\begin{align}
&\int_{\Rd}\Big((G(t,r+h)f)(x) -(G(t,r)f)(x)\Big)\mu_t(dx)\notag\\
=&
-\int_{\Rd}\bigg(\int_{r}^{r+h}(G(t,\sigma)\mathcal{A}(\sigma)f)(x)d\sigma\bigg)\mu_t(dx)\notag\\
=&- \int_{r}^{r+h}\bigg(\int_{\Rd}(G(t,\sigma)\mathcal{A}(\sigma)f)(x)\mu_t(dx)\bigg)d\sigma .
\label{form_KLL}
\end{align}
Using \eqref{invariant-0}, \eqref{form_KLL} can be rewritten as
\begin{align*}
\int_{\Rd}f(x)\mu_{r+h}(dx)-\int_{\Rd}f(x)\mu_r(dx)= -\int_{r}^{r+h}\int_{\Rd}
\Big((\mathcal{A}(\sigma)f)(x)\mu_\sigma(dx)\Big)d\sigma.
\end{align*}
Since the function $\sigma \mapsto
\int_{\Rd}(\mathcal{A}(\sigma)f)(x)\mu_{\sigma}(dx)=\int_{\Rd}(G(t,\sigma)\mathcal{A}(\sigma)f)(x)\mu_t(dx)$
is continuous in $I_t$, the claim follows dividing both sides by $h$ and letting $h\to 0$.

\vspace{3mm}

(ii) For $r$, $r+h\in [a,b]$ we have
\begin{align*}
&\int_{\Rd}f(r+h,x)\mu_{r+h}(dx)-\int_{\Rd}f(r,x)\mu_r(dx)\\
=& \int_{\Rd}\big(f(r+h,x)-f(r,x)\big)\mu_{r+h}(dx)+ \int_{\Rd}f(r,x)\mu_{r+h}(dx)-\int_{\Rd}f(r,x)\mu_r(dx).
\end{align*}
The statement  follows from (i) and from the continuity of the density $\rho$ in $I\times\Rd$ (Lemma \ref{density_mus}).
\end{proof}

In the proof of the LSI we will use also the next convergence lemma, a consequence of Lemma \ref{norma_p}.

\begin{prop}
For every $f \in C_b(\Rd)$ with positive infimum,
\begin{equation}\label{inf_fp}
\lim_{t\to+\infty}\int_{\Rd}(G(t,s)f)(x)\log((G(t,s)f)(x))\mu_t(dx)= m_s(f)
\log (m_s(f)),\qquad\;\, s \in I.
\end{equation}
\end{prop}
\begin{proof}
Since $G(t,s)$ preserves boundedness and positivity,  $G(t,s)f$ is bounded and has positive values, for every $t>s$.
Recalling that the function $y \mapsto y\log y$ is $1/2$-H\"older continuous on bounded sets of $[0,+\infty)$,  we get
\begin{align*}
&\left|\int_{\Rd}(G(t,s)f)(x)\log((G(t,s)f)(x))\mu_t(dx)- m_s(f)
\log (m_s(f))\right|\\
&\quad\quad=\left|\int_{\Rd}\left((G(t,s)f)(x)
\log((G(t,s)f)(x))- m_s(f) \log (m_s(f))\right)\mu_t(dx)\right|\\
&\quad\quad\leq C\int_{\Rd}|(G(t,s)f)(x)-m_s(f)|^{1/2}\mu_t(dx),
\end{align*}
for some positive constant $C$. By the H\"older  inequality,
\begin{eqnarray*}
\int_{\Rd}|(G(t,s)f)(x)-m_s(f)|^{1/2}\mu_t(dx) \leq \Vert G(t,s)f-
m_s(f)\Vert_{L^1(\R^d,\mu_t)}^{1/2}.
\end{eqnarray*}
Then, the claim follows from Lemma \ref{norma_p}.
\end{proof}

Now, we establish a logarithmic Sobolev inequality.

\begin{thm}\label{thm_log_sob}
For every $f\in C^1_b(\Rd)$,   $p\in (1,+\infty)$ and   $s\in I$, we have
\begin{align}
\int_{\Rd} |f|^p\log(|f|) \,\mu_s(dx)\leq &\frac{1}{p}m_s(|f|^p)
\log(m_s(|f|^p))\notag\\
&+\frac{p\Lambda}{2 |r_0|}\int_{\Rd}|f|^{p-2}|\nabla f|^2 \chi_{\{f\neq 0\}}\mu_s(dx).
\label{Log_Sob}
\end{align}
\end{thm}

\begin{proof}
To achieve \eqref{Log_Sob}, we would like to follow the method of Deuschel and Stroock:
differentiate the function
\begin{eqnarray*}
F(t):= \int_{\Rd}(G(t,s)f^p)(x)\log ((G(t,s)f^p)(x))\mu_t(dx),\qquad\;\,t \ge s,
\end{eqnarray*}
and prove that its derivative satisfies the inequality
\begin{equation}
F'(t)\ge -C e^{-c(t-s)} \int_{\Rd} f^{p-2}|\nabla f|^2 \mu_s(dx),\qquad\;\,t \ge s,
\label{star1}
\end{equation}
for some positive constants $C$ and $c$, independent of $f$. Then, the claim would follow by integrating
\eqref{star1} with respect to $t$ from $s$ to $+\infty$ and taking \eqref{inf_fp} into account.
However, we have to overcome some  difficulties due to the explicit time dependence of $\mu_t$. By Lemma \ref{derivative} we can differentiate the function
$\int_{\Rd}g \mu_t(dx)$ if $g$ is (smooth enough and) constant outside a compact set. But in general  $G(t,s)f^p\log(G(t,s)f^p)$ is not constant outside any compact set. Then we have to introduce a sequence of cut-off functions $\theta_n$ in the integral that defines $F$, and this gives rise to several additional terms that have to be controlled.

\vspace{3mm}
We split the proof in two steps. In the first step we prove \eqref{Log_Sob}
for functions  $f\in C^1_b(\Rd)$ with positive infimum, then we extend the claim
to general functions belonging to $C^1_b(\Rd)$.

\vspace{3mm}

{\em Step 1.}
Without loss of generality we may assume that $\sup f \leq 1$. Indeed,
for a general function $f \in C^1_b(\Rd)$ with positive infimum, the claim follows applying \eqref{Log_Sob} to the function $g=\frac{f}{\Vert f\Vert_\infty}$. Having $\sup f \leq 1$, we get $(G(t,s)f^p)(x)\log((G(t,s)f^p)(x))$ $\leq 0$ for $t>s$ and $x\in \R^d$, and this will be useful to control one of the additional terms coming from the cut-off functions.

So, let  $f \in C^1_b(\Rd)$  be such that $0<\delta \leq f(x)\leq 1$ for each $x\in \R^d$, and fix $s\in I$, $p>1$. Then,  \eqref{rep_nucleo} implies
$(G(t,s)f^p)(x)\in [\delta^p,1]$ for every $t\ge s$ and $x \in \Rd$.

The above mentioned cut-off functions are  standard. We fix
$\eta \in C^\infty(\R)$ satisfying $\chi_{(-\infty,1]}\leq \eta \leq \chi_{(-\infty,2]}$
and we set
\begin{equation}
\label{thetan}
\theta_n(x)=\eta\left(\frac{|x|}{n}\right), \qquad\;\,x\in \R^d, \;\,n\in \N,
\end{equation}
\begin{eqnarray*}
F_n(t)= \int_{\Rd}\theta_n(x)(G(t,s)f^p)(x)\log((G(t,s)f^p)(x))\mu_t(dx),\qquad\;\,t\ge s.
\end{eqnarray*}
For every $t\geq s$, $F_n(t)$ converges to
$F(t)$ as $n\to +\infty$, by dominated convergence. Moreover,  the function $(t,x)\mapsto \theta_n(x)(G(t,s)f^p)(x)\log((G(t,s)f^p)(x))$ is continuous and bounded in $[s, +\infty)\times\Rd$, and it satisfies the hypotheses of Lemma \ref{derivative}(ii) for any interval $[a,b]\subset (s, +\infty)$. Then,
   Lemma \ref{density_mus} and   Lemma
 \ref{derivative}(ii) yield   that  $F_n$ is continuous in $[s, +\infty)$ and  differentiable in $(s, +\infty)$, respectively. After
a long but straightforward computation, we get
\begin{align}
\frac{d}{dt}F_n(t)= & -\int_{\Rd}\theta_n \frac{\langle Q(t)\nabla_x G(t,s)f^p,
\nabla_x G(t,s)f^p\rangle}{G(t,s)f^p}\mu_t(dx)\notag\\
&-\int_{\Rd}(G(t,s)f^p)\log (G(t,s)f^p){\rm Tr}(Q(t)D^2 \theta_n)\mu_t(dx)\notag\\
&-\int_{\Rd}(G(t,s)f^p)\log (G(t,s)f^p) \langle b(t,\cdot), \nabla \theta_n\rangle \,\mu_t(dx)\notag\\
&- 2\int_{\Rd}\langle Q(t)\nabla \theta_n, \nabla_x G(t,s)f^p\rangle(\log(G(t,s)f^p)+ 1)\mu_t(dx)\notag\\[1mm]
=: & I_1(t,n) + I_2(t, n) + I_3(t, n) + I_4(t, n), \qquad\;\, t>s.
\label{natale}
\end{align}
Then, since $I_k(\cdot,n)$ $(k=1,\ldots,4)$ are bounded in $(s,t)$ for any $t>s$
and $F_n$ is continuous in $[s,+\infty)$,
\begin{equation}
F_n(t) - F_n(s) =  \sum_{k=1}^4 \int_s^tI_k(\sigma, n)\,d\sigma, \qquad\;\, t\geq s.
\label{inequalityn}
\end{equation}
We claim that
\begin{equation}
F(t)-F(s)\ge  \int_{s}^{t} g(\sigma) d\sigma,\qquad\;\,t\geq s,
\label{inequality}
\end{equation}
where
\begin{eqnarray*}
g(\sigma)= - \int_{\R^d} \frac{\langle Q(\sigma)\nabla_x G(\sigma,s)f^p,\nabla_x G(\sigma, s)f^p\rangle}{G(\sigma,s)f^p}\mu_{\sigma}(dx) ,\qquad\;\,\sigma\geq s.
\end{eqnarray*}
By \eqref{grad_est_punt} we have
\begin{align*}
|I_1(\sigma,n)-g(\sigma)|=&\left|\int_{\Rd}(\theta_n-1) \frac{\langle Q(\sigma)\nabla_x G(\sigma,s)f^p, \nabla_x G(\sigma,s)f^p\rangle}{G(\sigma,s)f^p}\mu_\sigma(dx)\right|\\
\leq &\Lambda \int_{\Rd}|\theta_n-1|\frac{|\nabla_x G(\sigma,s)f^p|^2}{G(\sigma,s)f^p}\mu_\sigma(dx)\\
\leq &\Lambda e^{2r_0(\sigma-s)}\int_{\Rd}|\theta_n-1|\frac{(G(\sigma,s)|\nabla f^p|)^2}{G(\sigma,s)f^p}\mu_\sigma(dx),
\end{align*}
for every  $\sigma\geq s$ and  $n\in\N$. The H\"older inequality and formula \eqref{rep_nucleo} imply
\begin{eqnarray*}
G(\sigma,s)|\nabla f^p|\leq \left(G(\sigma,s)\frac{|\nabla f^p|^2}{f^p}\right)^{1/2}(G(\sigma,s)f^p)^{1/2},\qquad\;\, \sigma\geq s.
\end{eqnarray*}
Thus,
\begin{align*}
|I_1(\sigma,n)-g(\sigma)|\le &\Lambda e^{2r_0(\sigma-s)}\int_{\Rd}|\theta_n-1|G(\sigma,s)\left(\frac{|\nabla f^p|^2}{f^p}\right)\mu_\sigma(dx)\\
\le &\Lambda e^{2r_0(\sigma-s)}\left\Vert\frac{|\nabla f^p|^2}{f^p}\right\Vert_\infty\int_{\Rd}|\theta_n-1|\mu_\sigma(dx),
\end{align*}
and consequently $\lim_{n\to +\infty} I_1(\sigma,n)=g(\sigma)$ for every  $\sigma\geq s$.
Moreover, $|I_1(\sigma,n)|\leq \Lambda p^2\Vert f^{p-2}|\nabla f|^2\Vert_\infty$. Integrating between $s$ and $t$, by dominated convergence we obtain
\begin{equation}
\label{I1}
\lim_{n\to +\infty}  \int_s^t I_1(\sigma,n)d\sigma =  \int_s^t g(\sigma)d\sigma ,\qquad\;\,t\geq s.
\end{equation}
Let us consider $I_2(\cdot , n)$. For $\sigma  \in I$ and $x\in \R^d$ we have
\begin{align*}
{\rm Tr}(Q(\sigma)D^2\theta_n(x)) =&\eta''\left(\frac{|x|}{n}\right)\frac{\langle Q(\sigma)x,x \rangle}{n^2|x|^2}
+\eta'\left(\frac{|x|}{n}\right)\frac{\textrm{Tr}(Q(\sigma))}{n |x|}\\
&- \eta'\left(\frac{|x|}{n}\right)\frac{\langle Q(\sigma)x,x \rangle}{n |x|^3 }.
\end{align*}
Recalling that   the supports of $\eta'$ and $\eta''$ are contained in $[1,2]$, we get
\begin{eqnarray*}
|{\rm Tr}(Q(\sigma)D^2\theta_n(x))| \leq   \frac{C_1}{n^2},\qquad\;\, \sigma \in I,
\end{eqnarray*}
where $C_1= d \Lambda(2\Vert \eta'\Vert_\infty+\Vert \eta''\Vert_\infty)$. Then,
\begin{eqnarray*}
|I_2(\sigma,n)|\le p|\log(\delta)| \frac{C_1}{n^2} ,\qquad\;\,\sigma\geq s,\;\,n\in\N,
\end{eqnarray*}
which implies
\begin{equation}
\label{I2}
\lim_{n\to +\infty}  \int_s^t I_2(\sigma,n)d\sigma =  0 ,\qquad\;\,t\geq s.
\end{equation}
Fix now $T>s$ and consider $I_3(\sigma, n)$ for $s\leq \sigma\leq T$. Again, since the support of $\eta'$ is contained in $[1, 2]$ and $\eta' \leq 0$, for every $x\in \R^d$ we have
\begin{equation}
\label{b0}
\langle b(\sigma,x),\nabla \theta_n(x)\rangle =  \eta'\left(\frac{|x|}{n}\right)\frac{\langle b(\sigma,x),x\rangle}{n|x|}
\geq  \eta'\left(\frac{|x|}{n}\right)\frac{C_{s,T}}{n^2} ,
\end{equation}
where $C_{s,T}$ is the constant in \eqref{inequality-b}. Recalling that   $ - (G(\sigma,s)f^p)\log (G(\sigma,s)f^p) \geq 0$ because $f^p\leq 1$,  we obtain
\begin{eqnarray*}
I_3(\sigma, n) \geq  - \int_{\R^d} (G(\sigma,s)f^p)\log (G(\sigma,s)f^p) \eta'\left(\frac{|x|}{n}\right)\frac{C_{s,T}}{n^2} \mu_\sigma(dx),
\qquad\;\, s\leq \sigma\leq T.
\end{eqnarray*}
On the other hand,
\begin{eqnarray*}
\bigg|  (G(\sigma,s)f^p)\log (G(\sigma,s)f^p) \eta'\left(\frac{|x|}{n}\right)\frac{C_{s,T}}{n^2}  \bigg|
 \leq p|\log \delta|   \frac{ C_{s, T}\|\eta '\|_{\infty}}{n^2} := \frac{C_2}{n^2},
 \end{eqnarray*}
and therefore
\begin{eqnarray*}
\liminf_{n\to +\infty}  \int_s^t I_3(\sigma,n)d\sigma \geq   0 ,\qquad\;\,s\leq t\leq T,
\end{eqnarray*}
and since $T$ is arbitrary,
\begin{equation}
\label{I3}
\liminf_{n\to +\infty}  \int_s^t I_3(\sigma,n)d\sigma \geq   0,\qquad\;\,t\geq s.
\end{equation}
$I_4(\cdot,n)$ tends to $0$ as $n\to +\infty$, uniformly in $[s, +\infty)$, since
  \eqref{grad_est_punt} yields
\begin{align*}
|I_4(\sigma ,n)|\leq &
2\Lambda \frac{\Vert \eta'\Vert_\infty}{n}e^{r_0(\sigma-s)}\int_{\Rd}(G(\sigma ,s)|\nabla f^p|)|\log(G(\sigma,s)f^p)+ 1|\mu_{\sigma}(dx)\\
\le & 2\Lambda\frac{\Vert \eta'\Vert_\infty}{n}
\Vert\,|\nabla f^p|\,\Vert_\infty\left(p|\log(\delta)|+1\right),
\end{align*}
for every $\sigma \geq s$. Hence,
\begin{equation}
\label{I4}
\lim_{n\to +\infty}  \int_s^t I_4(\sigma,n)d\sigma =   0,\qquad\;\,t\geq s.
\end{equation}

Taking into account \eqref{I1}, \eqref{I2}, \eqref{I3}, \eqref{I4} and letting $n\to +\infty$ in \eqref{inequalityn}, formula \eqref{inequality} follows.
Now, since
\begin{align*}
\int_{\Rd}\frac{\langle Q(\sigma)\nabla_x G(\sigma,s)f^p, \nabla_x G(\sigma,s)f^p\rangle}{G(\sigma,s)f^p}\mu_\sigma(dx)&
\leq \Lambda e^{2r_0(\sigma-s)}\int_{\Rd}G(\sigma,s)\frac{|\nabla f^p|^2}{f^p}\mu_\sigma(dx)\\
& =\Lambda p^2 e^{2r_0(\sigma-s)} \int_{\Rd} f^{p-2}|\nabla f|^2 \mu_s(dx),
\end{align*}
we get, for $t\geq s$,
\begin{align*}
F(t)-F(s)&\geq -\Lambda p^2 \int_{\Rd} f^{p-2}|\nabla f|^2 \mu_s(dx) \int_{s}^{t} e^{2r_0(\sigma-s)}d\sigma\nonumber\\
&=  \frac{\Lambda p^2}{2r_0}\left(1 - e^{2r_0(t-s)} \right) \int_{\Rd} f^{p-2}|\nabla f|^2 \mu_s(dx).
\end{align*}
Letting $t\to +\infty$ and recalling  \eqref{inf_fp} yields
\begin{eqnarray*}
m_s(f^p)  \log (m_s(f^p) )-F(s)\geq
\frac{\Lambda p^2}{2r_0}\int_{\Rd} f^{p-2}|\nabla f|^2 \mu_s(dx),
\end{eqnarray*}
that is,
\begin{eqnarray*}
\int_{\Rd} f^p\log f^p \,\mu_s(dx)\leq m_s(f^p) \log(m_s(f^p))+
\frac{\Lambda p^2}{2|r_0|}\int_{\Rd} f^{p-2}|\nabla f|^2 \mu_s(dx),
\end{eqnarray*}
which coincides with \eqref{Log_Sob} in our case.

\vspace{3mm}

{\em Step 2.}
Let now $f \in C^1_b(\Rd)$ and define
$f_n:=(f^2+n^{-1})^{1/2}$.
By the first part of the proof we have
\begin{equation}
\label{Log_Sob_n}
\int_{\Rd} f_n^p\log(f_n^p) \,\mu_s(dx)\leq m_s(f_n^p)\log(m_s(f_n^p))+
\frac{\Lambda p^2}{2 |r_0|}\int_{\Rd}f_n^{p-2}|\nabla f_n|^2\mu_s(dx),
\end{equation}
for any $n \in \N$ and $s \in I$.
Since $0< f_n^p \leq \Vert f^2+1\Vert_{\infty}^{p/2}$,  the left-hand
side of \eqref{Log_Sob_n} converges to $\int_{\Rd} |f|^p\log |f|^p \,\mu_s(dx)$.
Similarly,
by   dominated convergence   we obtain
\begin{eqnarray*}
\lim_{n \to +\infty}m_s(f_n^p) \log(m_s(f_n^p))=
m_s(|f|^p)\log(m_s(|f|^p)).
\end{eqnarray*}
Observe that
$|\nabla f_n|^2\le |\nabla f|^2$ for any $n\in\N$; by the monotone convergence theorem, if $p<2$, and by dominated convergence, otherwise, we   get
\begin{eqnarray*}
\lim_{n \to +\infty}\int_{\Rd}f_n^{p-2}|\nabla f_n|^2\mu_s(dx)=
\int_{\Rd}|f|^{p-2}|\nabla f|^2 \chi_{\{f \neq 0\}}\mu_s(dx),
\end{eqnarray*}
and the statement follows letting $n \to +\infty$ in \eqref{Log_Sob_n}.
\end{proof}

The logarithmic Sobolev inequality \eqref{Log_Sob} yields some compactness results.

\begin{thm}\label{prop-compatt}
Fix $s\in I$. Then:
\begin{enumerate}[\rm (i)]
\item
$W^{1,p}(\Rd,\mu_s)$ is compactly embedded in $L^p(\R^d,\mu_s)$ for any $p\in [2,+\infty)$;
\item
for any $t>s$ and $p\in (1,+\infty)$, the operator $G(t,s):L^p(\Rd,\mu_s)\to L^p(\Rd, \mu_t)$ is compact.
\end{enumerate}
\end{thm}

\begin{proof}
(i) Let $\mu$ be a Borel measure in $\R^d$ such that for every $R>0$ and $p\geq 2$,
$L^p(B(0,R),\mu)=L^p(B(0,R),dx)$ with equivalence of the corresponding norms (which is true for our measures $\mu_s$).
It is known that the occurrence of a logarithmic Sobolev inequality for  $\mu$  yields compactness of the embedding $W^{1,2}(\Rd,\mu)\subset L^2(\Rd,\mu)$, see e.g., \cite{LMP}. The proof for $p\geq 2$ is not much different, we write it here for the reader's convenience.

Let ${\mathcal B}$ be a ball in $W^{1,p}(\Rd,\mu_s)$. We prove that ${\mathcal B}$ is totally bounded in $L^p(\R^d,\mu_s)$. For this purpose, we fix
$\varepsilon>0$ and claim that there exists $R>0$ such that
\begin{equation}
\|f\|_{L^p(\R^d\setminus B(0,R),\mu_s)}\le \frac{\varepsilon}{2},\qquad\;\,f\in \mathcal{B}.
\label{A}
\end{equation}
For any fixed $f\in {\mathcal B}$ and $k\in\N$ we introduce the set $E_k=\{x\in\Rd: |f(x)|\le k\}$. Then,
\begin{align*}
\int_{\Rd\setminus B(0,R)}|f|^pd\mu_s=&\int_{E_k\cap (\Rd\setminus B(0,R))}|f|^pd\mu_s
+\int_{\Rd\setminus (B(0,R)\cup E_k)}|f|^pd\mu_s\\
\le & k^p\mu_s(\Rd\setminus B(0,R))
+\frac{1}{\log(k)}\int_{\Rd}|f|^p\log(|f|)d\mu_s.
\end{align*}
By the logarithmic Sobolev inequality \eqref{Log_Sob} (which can be extended to any
$g\in W^{1,p}(\R^d,\mu_s)$ by Lemma \ref{density}, since $p\ge 2$) and the H\"older inequality  we obtain
\begin{align*}
\int_{\Rd}|f|^p\log(|f|)d\mu_s
\le & \|f\|_{L^p(\Rd,\mu_s)}^p\log(\|f\|_{L^p(\Rd,\mu_s)})\\
&+\frac{p\Lambda}{2|r_0|}
\|f\|_{L^p(\Rd,\mu_s)}^{p-2}\|\,|\nabla f|\,\|_{L^p(\Rd,\mu_s)}^2\\
\le & K,
\end{align*}
for some $K>0$. Therefore,
\begin{eqnarray*}
\int_{\Rd\setminus B(0,R)}|f|^pd\mu_s
\le   k^p\mu_s(\Rd\setminus B(0,R))
+\frac{K}{\log k },\qquad\;\,f\in {\mathcal B}.
\end{eqnarray*}
The claim follows choosing properly $k$ and $R$.

By Lemma \ref{density_mus} the density of $\mu_s$ with respect to the Lebesgue measure is a continuous and  positive function. Since, as we have already recalled the spaces $L^p(B(0,R),\mu_s)$ and $L^p(B(0,R),dx)$ (and, hence, the spaces
$W^{1,p}(B(0,R),\mu_s)$ and $W^{1,p}(B(0,R),dx)$) coincide, and the corresponding norms are equivalent,
by the Rellich-Kondrachov theorem there exists a finite number of functions $f_1,\ldots,f_m$ in $L^p(B(0,R),\mu_s)$ such that
\begin{eqnarray*}
\mathcal{B}_{|B(0,R)}\subset \bigcup_{j=1}^m\left\{f\in L^p(B(0,R),\mu_s): \|f-f_j\|_{L^p(B(0,R),\mu_s)}\le \frac{\varepsilon}{2}\right\},
\end{eqnarray*}
where $\mathcal{B}_{|B(0,R)}$ denotes the set of   the restrictions to $B(0,R)$ of the functions in $\mathcal{B}$.
Using \eqref{A} it is easy to check that $\mathcal{B}$ is contained in the union of the closed balls in $L^p(\Rd,\mu_s)$ centered at $\widetilde{f_j}$, with radius $\varepsilon$, where $\widetilde{f_j}$ denotes the trivial extension of $f_j$ to the whole of $\Rd$.

\vspace{3mm}

(ii) The proof follows by interpolation. Indeed, by estimate \eqref{grad_est_norm_p}, $G(t,s)$ maps $L^p(\Rd,\mu_s)$ to $W^{1,p}(\Rd,\mu_t)$ and $W^{1,p}(\Rd,\mu_t)$ is compactly embedded in $L^p(\Rd,\mu_t)$, for every $p \ge 2$. Hence, $G(t,s)$ is a compact operator from $L^p(\Rd,\mu_s)$ to $L^p(\Rd,\mu_t)$ for every $p \ge 2$. Moreover, $G(t,s)$ is also a linear bounded operator from $L^1(\Rd,\mu_s)$ to $L^1(\Rd,\mu_t)$. Then, the claim follows arguing as in \cite[Thm. 1.6.1]{Davies} with $A=G(t,s)$ and with obvious modifications.
\end{proof}

\section{Hypercontractivity of $G(t,s)$}\label{sec_hyper}

The LSI inequality is the main tool in the proof of the following hypercontractivity theorem.

\begin{thm}\label{thm_fin}
Let $s \in I$,
$p,q\in (1,+\infty)$ with $p\leq e^{2\eta_0 |r_0|\Lambda^{-1}(t-s)}(q-1)+1$.
Then,  $G(t,s)$ maps $L^q(\Rd,\mu_s)$ to $L^{p}(\Rd, \mu_t)$
for every $t>s$ and
\begin{equation}
\Vert G(t,s)f \Vert_{L^{p}(\Rd,\mu_t)}\leq \Vert f\Vert_{L^q(\Rd,\mu_s)}, \qquad\;\, t>s, \;\, f \in L^q(\Rd,\mu_s).
\label{hyper-estim}
\end{equation}
\end{thm}
\begin{proof} The proof is in two steps.  In the first step we show that \eqref{hyper-estim}
holds for every positive $f\in C^{\infty}_b(\Rd)$, which is constant outside a compact subset of $\Rd$. In the second step we extend \eqref{hyper-estim} to all $f\in L^q(\Rd,\mu_s)$.

\vspace{3mm}

{\em Step 1.} Let $f\in C^{\infty}_b(\Rd)$ be  constant outside a compact subset of $\Rd$ and have positive values.
Fix $q>1$, $s\in I$ and set  $p(t):=e^{2\eta_0 |r_0|\Lambda^{-1}(t-s)}(q-1)+1$. Our aim is to show that
the function
\begin{eqnarray*}
\beta(t):= \Vert G(t,s)f\Vert_{L^{p(t)}(\Rd,\mu_t)},\qquad\;\,t\ge s,
\end{eqnarray*}
is decreasing.  This will imply $\|G(t,s)f \|_{L^{p(t)}(\Rd,\mu_t)}\leq \|f\|_{L^q(\Rd,\mu_s)}$, and for $p<p(t)$ \eqref{hyper-estim} will follow from the H\"older inequality.

Let $\theta_n$ be the cut-off functions defined in \eqref{thetan}.  By \cite[Thm. 2.2, Step 1]{KunLorLun09Non}, $(t,x)\mapsto (G(t,s)f)(x)\in C^{1,2}([s, +\infty)\times \R^d)$. Then,  Lemma \ref{derivative}(ii) yields that  the function
\begin{eqnarray*}
t\mapsto \int_{\Rd} \theta_n (G(t,s)f)^{p(t)} \mu_t(dx), \qquad\;\, t\geq s,
\end{eqnarray*}
is differentiable in $[s, +\infty)$ for every $n\in \N$, with derivative given by
\begin{align}
&p'(t)
\int_{\Rd} \theta (G(t,s)f)^{p(t)}\log (G(t,s)f)\mu_t(dx)\notag\\
&- p(t)(p(t)-1)\int_{\Rd} \theta (G(t,s)f)^{p(t)-2}
\langle Q(t)\nabla_x G(t,s)f,\nabla_x G(t,s)f\rangle \mu_t(dx)\notag\\
&- 2 \int_{\Rd} \langle Q(t) \nabla_x ((G(t,s)f)^{p(t)}), \nabla \theta\rangle \mu_t(dx)
-\int_{\Rd} (G(t,s)f)^{p(t)}\mathcal{A}(t)\theta \mu_t(dx).
\label{parziale}
\end{align}
Let us define the functions $\beta_n$ by
\begin{eqnarray*}
\beta_n(t):=\left(\int_{\Rd}\theta_n(x)((G(t,s)f)(x))^{p(t)}\mu_t(dx)\right)^{\frac{1}{p(t)}},\qquad\;\,t\in [s,+\infty).
\end{eqnarray*}
Then,
\begin{align*}
\beta'_n(t)=  \beta_n(t)
\bigg [&-\frac{p'(t)}{(p(t))^2}\log\left(\int_{\Rd}\theta_n(x)((G(t,s)f)(x))^{p(t)}\mu_t(dx)\right)\nonumber\\
&+\frac{1}{p(t)}\beta_n(t)^{-p(t)}
\frac{d}{dt}\left(\int_{\Rd}\theta_n(x)((G(t,s)f)(x))^{p(t)}\mu_t(dx)\right) \bigg ]. \nonumber
\end{align*}
Replacing \eqref{parziale} we get
\begin{align*}
\beta'_n(t)  = & \beta_n(t)\bigg\{-\frac{p'(t)}{(p(t))^2}\log\left (\int_{\Rd}\theta_n (G(t,s)f)^{p(t)}\mu_t(dx)\right)\\
&+\frac{1}{p(t)(\beta_n(t))^{p(t)}}\bigg [p'(t)\int_{\Rd}\theta_n(G(t,s)f)^{p(t)}\log (G(t,s)f)\mu_t(dx)\\
&- p(t)(p(t)-1)\int_{\Rd}\theta_n(G(t,s)f)^{p(t)-2}\langle Q(t)\nabla_x G(t,s)f,\nabla_x G(t,s)f \rangle\mu_t(dx)\\
&-2\int_{\Rd}\langle Q(t)\nabla_x((G(t,s)f)^{p(t)}),\nabla\theta_n\rangle\mu_t(dx)\\
&-\int_{\Rd}(G(t,s)f)^{p(t)}{\mathcal A}(t)\theta_n\mu_t(dx)\bigg ] \bigg\} .
\end{align*}

Let us  fix $T>s$ and note that by \eqref{b0}  we have
\begin{eqnarray*}
(\mathcal{A}(t)\theta_n)(x) \ge {\rm Tr}(Q(t)D^2\theta_n(x)) + \eta'\left(\frac{|x|}{n}\right)\frac{C_{s,T}}{n^2}, \qquad\;\, s\leq t\leq T, \;\,x\in \R^d.
\end{eqnarray*}
Hence,  for $s\leq t\leq T$,
\begin{align*}
\beta'_n&(t)\le  \gamma_n(t):=\beta_n(t)\bigg\{ -\frac{p'(t)}{p^2(t)}\log\left (\int_{\Rd}\theta_n(G(t,s)f)^{p(t)}\mu_t(dx)\right )\nonumber\\
&+\frac{1}{p(t)(\beta_n(t))^{p(t)}}\bigg [p'(t)\int_{\Rd}\theta_n(G(t,s)f)^{p(t)}\log (G(t,s)f)\mu_t(dx)\\
&- p(t)(p(t)-1)\int_{\Rd}\theta_n(G(t,s)f)^{p(t)-2}\langle Q(t)\nabla_x G(t,s)f,\nabla_x G(t,s)f \rangle \mu_t(dx)\\
&-2\int_{\Rd}\langle Q(t)\nabla_x((G(t,s)f)^{p(t)}),\nabla\theta_n\rangle\mu_t(dx)\\
&-\int_{\Rd}((G(t,s)f )(x))^{p(t)}\left ({\rm Tr}(Q(t)D^2\theta_n(x))
+ \eta'\left(\frac{|x|}{n}\right)\frac{C_{s,T}}{n^2}\right )\mu_t(dx)\bigg ] \bigg\}.
\end{align*}
Therefore, for   $s\le t_1<t_2\le T$ we have
\begin{equation}
\beta_n(t_2)-\beta_n(t_1)\le\int_{t_1}^{t_2}\gamma_n(s)ds.
\label{form-1}
\end{equation}
Our claim will be proved letting $n\to +\infty$ in \eqref{form-1}. To this aim we note that
\begin{align*}
|\beta_n(t)-\beta(t)|\le &\|(\theta_n-1)G(t,s)f\|_{L^{p(t)}(\Rd,\mu_t)}
\le\|f\|_{\infty}\left (\int_{\Rd}|\theta_n-1|\mu_t(dx)\right )^{1/p(t)},
\end{align*}
which shows that $\beta_n(t)$ tends to $\beta(t)$ for   $t\in [s,T]$, as $n\to +\infty$. Moreover,
$|\beta_n(t)|\le \|f\|_{\infty}$ for   $t\geq s$ and   $n\in\N$.

Let us prove  that $\gamma_n$ converges pointwise in $[s,T]$ to the function $\gamma$ defined by
\begin{align*}
\gamma(t):=&
\beta(t)\bigg[-\frac{p'(t)}{p^2(t)}\log\left (\int_{\Rd}(G(t,s)f)^{p(t)}\mu_t(dx)\right )\nonumber\\
&+\frac{1}{p(t)(\beta(t))^{p(t)}}\bigg(p'(t)\int_{\Rd}(G(t,s)f)^{p(t)}\log (G(t,s)f)\mu_t(dx)\\
&- p(t)(p(t)-1)\int_{\Rd}(G(t,s)f)^{p(t)-2}\langle Q(t)\nabla_x G(t,s)f,\nabla_x G(t,s)f \rangle \mu_t(dx)\bigg) \bigg],
\end{align*}
and that there exists a positive constant $C_1$ such that $|\gamma_n(t)|\le C_1$ for each $t\in [s,T]$ and   $n\in\N$. We have to discuss convergence and estimates just for
\begin{eqnarray*}
I_{1,n}(t):=\int_{\Rd}\theta_n (G(t,s)f)^{p(t)}\log (G(t,s)f) \mu_t(dx)
\end{eqnarray*}
and
\begin{eqnarray*}
I_{2,n}(t):=\int_{\Rd}\theta_n(G(t,s)f)^{p(t)-2}\langle Q(t)\nabla_x G(t,s)f,\nabla_x G(t,s)f \rangle \mu_t(dx),
\end{eqnarray*}
since the other terms are easier to deal with: it is enough to recall that $\|\,|\nabla\theta_n|\,\|_{\infty}\le Cn^{-1}$ and $|{\rm Tr}(Q(t)D^2\theta_n(x))| \leq  C_1 n^{-2}$, $| \eta' (|x|/n) C_{s,T} n^{-2}| \leq \|\eta '\|_{\infty} C_{s,T} n^{-2}$,  as it has been already done in the proof of Theorem \ref{thm_log_sob}.

Concerning $I_{1,n}(t)$, we observe that
\begin{eqnarray*}
0< (G(t,s)f)^{p(t)}\log (G(t,s)f)\leq \max\{ \xi^{p(t)}\log \xi:\;\, t\in [s,T],\; 0 <\xi \leq \|f\|_{\infty}\} =: H,
\end{eqnarray*}
and hence
\begin{eqnarray*}
\left |I_{1,n}(t)
-\int_{\Rd}(G(t,s)f)^{p(t)}\log (G(t,s)f) \mu_t(dx)\right |
\le H\int_{\R^d}|\theta_n-1| \mu_t(dx).
\end{eqnarray*}
The right hand side  vanishes as $n\to +\infty$, and it is bounded by $H$ for every $t\in [s,T]$.

Next, we consider $I_{2,n}(t)$
and we note that
\begin{align}\label{i2n}
&\left |I_{2,n}(t)-\int_{\Rd}(G(t,s)f)^{p(t)-2}\langle Q(t)\nabla_x G(t,s)f,\nabla_x G(t,s)f \rangle \mu_t(dx)\right |\nonumber\\
\le &\Lambda\int_{\Rd}|\theta_n-1|(G(t,s)f)^{p(t)-2}|\nabla_xG(t,s)f|^2\mu_t(dx).
\end{align}
Using \eqref{grad_est_punt} and, if $p(t) <2$, the inequality $G(t,s)f\ge G(t,s)(\inf f)=\inf f>0$, that follows from
\eqref{rep_nucleo}, we easily deduce that
the right-hand side of \eqref{i2n} vanishes as $n\to +\infty$. Moreover,
\begin{eqnarray*}
|I_{2,n}(t)|\le\Lambda\int_{\Rd}(G(t,s)f)^{p(t)-2}|\nabla_xG(t,s)f|^2\mu_t(dx),\qquad\;\,t\in [s,T],\;\,n\in\N.
\end{eqnarray*}

Then, we may let  $n\to +\infty$ in \eqref{form-1} and conclude that for $s\le t_1< t_2\le T$ we have
\begin{eqnarray*}
\beta(t_2)-\beta(t_1)\le \int_{t_1}^{t_2}\gamma(s)ds,
\end{eqnarray*}
and since $T>s$ is arbitrary, the inequality holds for every  $s\le t_1< t_2$.

Applying the logarithmic Sobolev
inequality \eqref{Log_Sob} to the function $G(t,s)f$ we get, for $t>s$,
\begin{align*}
\gamma(t)&\leq  \beta(t)^{1-p(t)}\bigg(\frac{p'(t)\Lambda}{2|r_0|}\int_{\Rd}(G(t,s)f)^{p(t)-2}|\nabla_x G(t,s)f|^2 \mu_t(dx)\\
&\quad -(p(t)-1)\int_{\Rd}(G(t,s)f)^{p(t)-2}\langle Q(t)\nabla_x G(t,s)f,\nabla_x G(t,s)f \rangle \mu_t(dx)\bigg)\\
& \leq \beta(t)^{1-p(t)}\bigg(\frac{p'(t)\Lambda}{2|r_0|}-(p(t)-1)\eta_0\bigg)\int_{\Rd}(G(t,s)f)^{p(t)-2}|\nabla_x G(t,s)f|^2 \mu_t(dx)\\
&=0,
\end{align*}
by the definition of $p(t)$. Hence, $\beta(t_2)\leq \beta(t_1)$.

\vspace{3mm}
{\em Step 2.}  Let   $f\in C^{\infty}_c(\Rd)$ and consider the sequence of functions
\begin{eqnarray*}
f_n(x)=\sqrt{|f(x)|^2+\frac{1}{n}},\qquad\;\,x\in\Rd,\;\,n\in\N.
\end{eqnarray*}
Each  $f_n$ belongs to $C^{\infty}_b(\Rd)$, it has positive values and it is  constant outside a compact subset of $\Rd$. By Step 1, for $p  \leq e^{2\eta_0 |r_0|\Lambda^{-1}(t-s)}(q-1)+1$ we have
\begin{eqnarray*}
\Vert G(t,s)f_n \Vert_{L^{p}(\Rd,\mu_t)}\leq \Vert f_n\Vert_{L^q(\Rd,\mu_s)},\qquad\;\,t\ge s.
\end{eqnarray*}
Since $f_n$ converges to $|f|$ uniformly in $\Rd$, $f_n$ and $G(t,s)f_n$ converge to
$|f|$ and $G(t,s)|f|$ in $L^q(\Rd,\mu_t)$ and in $L^{p }(\Rd,\mu_s)$, respectively, as $n\to +\infty$.
Therefore,
\begin{align}
\Vert G(t,s)f\Vert_{L^{p }(\Rd,\mu_t)}\le
\Vert G(t,s)|f|\Vert_{L^{p }(\Rd,\mu_t)}
\leq \Vert f\Vert_{L^q(\Rd,\mu_s)}.
\label{estim-1}
\end{align}
For a general $f\in L^q(\Rd,\mu_s)$ estimate \eqref{estim-1} follows by density,
approximating $f$ by a sequence $(f_n)_n\subset C^{\infty}_c(\Rd)$.\end{proof}


\section{Poincar\'e inequality and asymptotic behavior}
\label{sec_poincare}


This section is devoted to prove the Poincar\'e inequality for the measures $\mu_t$ and to the study of the decay rate of $ G(t,s) -m_s $ for  $p>1$.

The Poincar\'e  inequality \eqref{Poincare} could  be proved by contradiction, through a classical argument (see e.g., \cite[Thm. 5.8.1]{evans}) that exploits the compactness of the embedding
$W^{1,p}(\Rd,\mu_s)\hookrightarrow L^p(\Rd,\mu_s)$. However, this procedure does not allow  to control the dependence of the constant $C_p$ below
on $s$, whereas in the proof of the next Theorem \ref{thm-Ap-Bp} we need $C_2$ to be independent of $s$.
Hence, we use different arguments. In particular, we use the following lemma.

\begin{lemm}
\label{estensione_p}
Let $\mu(dx) = \rho(x)dx$ be a probability measure in $\R^d$, where $\rho$ is a locally bounded function with positive infimum on every ball, and denote by $m(f)$ the mean value of $f\in W^{1,p}(\Rd,\mu )$ with respect to $\mu$. Assume that a Poincar\'e inequality holds for $p=2$, that is
\begin{equation}\label{Poincare2}
\Vert f- m(f) \Vert_{L^2(\R^d,\mu )}\leq C \Vert\,|\nabla f|\,\Vert_{L^2(\R^d,\mu)},\qquad\;\, f\in W^{1,2}(\Rd,\mu).
\end{equation}
Then for every $p>2$ there is $C_p>0$, depending only on $C$ and $p$,  such that
\begin{equation}\label{Poincarep}
\Vert f- m(f)  \Vert_{L^p(\R^d,\mu )}\leq C_p\Vert\,|\nabla f|\,\Vert_{L^p(\R^d,\mu )}, \qquad\;\, f\in W^{1,p}(\Rd,\mu).
\end{equation}
\end{lemm}
\begin{proof}
 As a first step, we prove that there exists a positive constant $K_p$ such that
\begin{equation}
\|g\|_{L^p(\Rd,\mu )}^p\le
K_p\ \|\,|\nabla g|\,\|_{L^p(\Rd,\mu)}^p+2\|g\|_{L^{p/2}(\Rd,\mu)}^p,\qquad\;\, g\in W^{1,p}(\Rd,\mu).
\label{stima-partenza}
\end{equation}
Since $p>2$, for every $g\in W^{1,p}(\Rd,\mu)$ the function $|g|^{p/2}$ belongs to $W^{1,2}(\Rd,\mu)$. This can be proved approaching $g$ by a sequence of functions in $C^{\infty}_c(\R^d)$, which is dense in $W^{1,p}(\Rd,\mu)$ by Lemma \ref{density_mus}.

Applying the Poincar\'e inequality \eqref{Poincare2} to
 $|g|^{p/2}$ yields
\begin{eqnarray*}
\|g\|_{L^p(\Rd,\mu)}^p\le \frac{(pC)^2}{4 }\|g\|_{L^p(\Rd,\mu)}^{p-2}\|\,|\nabla g|\,\|_{L^p(\Rd,\mu)}^2
+\|g\|_{L^{p/2}(\Rd,\mu)}^p.
\end{eqnarray*}
Using the Young inequality $a^{1-2/p}b^{2/p}\le \varepsilon a+C_{\varepsilon,p}b$
with $a=\|g\|_{L^p(\Rd,\mu)}^p$, $b= \|\,|\nabla g|\,\|_{L^p(\Rd,\mu)}^p$,
and choosing $\varepsilon=2(pC)^{-2}$,
we get \eqref{stima-partenza}
for some positive constant $K_p$ depending only on $C$ and $p$.

If $p\in (2,4]$ we are done: since $p/2\leq 2$, we have $\|g\|_{L^{p/2}(\Rd,\mu)}\leq \|g\|_{L^{2}(\Rd,\mu)}$ and \eqref{stima-partenza} yields
\begin{equation}
\|g\|_{L^p(\Rd,\mu)}^p\le
K_p \|\,|\nabla g|\,\|_{L^p(\Rd,\mu)}^p+2\|g\|_{L^2(\Rd,\mu)}^p ,\qquad\;\, g\in W^{1,p}(\Rd,\mu).
\label{stima-partenza-1}
\end{equation}
Taking  $g = f-m(f)$ with any $f\in  W^{1,p}(\Rd,\mu)$ we get
\begin{align}
\|f-m(f)\|_{L^p(\Rd,\mu)}^p
\le K_p \|\,|\nabla f|\,\|_{L^p(\Rd,\mu)}^p+2\|f-m(f)\|_{L^2(\Rd,\mu)}^p .
\label{00}
\end{align}
Using again \eqref{Poincare2} and  $\|\,|\nabla f|\,\|_{L^2(\Rd,\mu)}\le \|\,|\nabla f|\,\|_{L^p(\Rd,\mu)}$,  \eqref{00} yields
\begin{equation}
\|f-m(f)\|_{L^p(\Rd,\mu)}\le \widetilde{K}_p\|\,|\nabla f|\,\|_{L^p(\Rd,\mu)}, \qquad\;\, f\in W^{1,p}(\Rd,\mu),\qquad\;\,2< p \leq 4,
\label{poincare-p-1}
\end{equation}
with some positive constant $\widetilde{K}_p$ depending only on $C$ and $p$, i.e.,
\eqref{Poincarep} holds true for $2<p\leq 4$.

Let  now $p\in [4,8)$. For any $f\in W^{1,p}(\Rd,\mu)$ we apply \eqref{stima-partenza} to the function
$g=f-m(f)$, and since   $p/2\in [2,4)$, we  may use \eqref{poincare-p-1}
 with $p/2$ instead of $p$,  to get
\eqref{Poincarep}.

Iterating this procedure, we get   \eqref{Poincarep} for any $p>2$.  \end{proof}

\begin{thm}\label{Poi_2}
For every $p\ge 2$, there exists a positive constant $C_p$ such that
\begin{equation}\label{Poincare}
\Vert f- m_s(f) \Vert_{L^p(\R^d,\mu_s)}\leq C_p\Vert\,|\nabla f|\,\Vert_{L^p(\R^d,\mu_s)},
\end{equation}
for any $f \in W^{1,p}(\Rd,\mu_s)$ and any $s \in I$. In particular, if $p=2$, we can take
$C_2=\Lambda^{1/2}|r_0|^{-1/2}$, where $\Lambda$ and $r_0$ are defined in Hypotheses $\ref{hyp1}$.
\end{thm}

\begin{proof}
For $p=2$, estimate \eqref{Poincare} follows from the LSI \eqref{Log_Sob} by \cite{rothaus}.
For the reader's convenience we give a sketch of the proof.
Let   $f \in C^1_b(\Rd)$ be such that  $m_s(f)=0$. Replacing  $f$ by  $f_{\varepsilon}:=1+\varepsilon f$
in the inequality \eqref{Log_Sob} with $p=2$ yields
\begin{align*}
\int_{\Rd} f_\varepsilon^2(\log f_\varepsilon^2)\mu_s(dx)-\|f_\varepsilon\|_{L^2(\Rd,\mu_s)}^2
\log(\|f_\varepsilon\|_{L^2(\Rd,\mu_s)}^2)
= 2\varepsilon^2\|f_\varepsilon\|_{L^2(\Rd,\mu_s)}^2+o(\varepsilon ^2),
\end{align*}
as $\varepsilon \to 0^+$.
Thus, by \eqref{Log_Sob} we get \eqref{Poincare}. If $m_s(f)\neq 0$ it suffices to consider the function
$f-m_s(f)$. Since
 $C^1_b(\Rd)$ is dense in $W^{1,2}(\Rd,\mu_s)$ (see Lemma \ref{density}) the claim follows.

For $p>2$, the statement follows from Lemma \ref{estensione_p}. \end{proof}

Next theorem shows how the decay of $G(t,s) -m_s $ to $0$  is related to the decay of $\nabla_x G(t,s)$ to $0$.
A similar result has been proved in \cite[Thm. 3.6]{LorLunZam10} in the case of time periodic coefficients under stronger assumptions than ours, and only for $p\ge 2$. The approach of  \cite{LorLunZam10} is different from the present one, since it relies on the use of  the evolution semigroup associated to the evolution family $G(t,s)$.

For $p\geq 1$ we define the sets $\mathfrak{A}_p$ and $\mathfrak{B}_p$ by
\begin{align*}
\mathfrak{A}_p=&\big\{\omega\in\R: \exists M_{p,\omega}>0~{\rm s.t.}~\|G(t,s)f-m_s(f)\|_{L^p(\R^d,\mu_t)}\le M_{p,\omega}e^{\omega (t-s)}\|f\|_{L^p(\R^d,\mu_s)},\\
&~~~\;I\ni s<t,\;\,f\in L^p(\R^d,\mu_s)\big\};\\[1mm]
\mathfrak{B}_p=&\big\{\omega\in\R: \exists N_{p,\omega}>0~{\rm s.t.}~\|\,|\nabla_x G(t,s)f|\,\|_{L^p(\R^d,\mu_t)}\le N_{p,\omega}e^{\omega ( t-s)}\|f\|_{L^p(\R^d,\mu_s)},\\
&~~~\;s,t\in I,\, t-s\ge 1,\;\,f\in L^p(\R^d,\mu_s) \big\}.
\end{align*}

\begin{thm}
\label{thm-Ap-Bp}
The sets ${\mathfrak A}_p$ and  $\mathfrak{B_p}$ are independent of $p\in (1,+\infty)$, and they coincide.
\end{thm}

\begin{proof}
As a first step we prove that $\mathfrak{A}_p$ and $\mathfrak{B}_p$ are independent of $p\in (1,+\infty)$.
Then,  we prove that $\mathfrak{A}_2=\mathfrak{B}_2$, which yields the conclusion.

\emph{Step 1.} To prove  that $\mathfrak{A}_p$ is independent of $p$ we use the hypercontractivity of  $G(t,s)$.
Fix $p_1>1$ and set $p_2=e^{2\eta_0|r_0|\Lambda^{-1}}(p_1-1)+1$. Clearly, $p_1<p_2$.
Let  $p\in (p_1,p_2]$. From \eqref{hyper-estim} we deduce
\begin{align*}
\|G(t,s)f-m_s(f)\|_{L^p(\Rd,\mu_t)}=&
\|G(t,t-1)(G(t-1,s)f-m_s(f))\|_{L^p(\Rd,\mu_t)}\\
\le &\|G(t-1,s)f-m_s(f)\|_{L^{p_1}(\Rd,\mu_{t-1})},
\end{align*}
for any $f\in L^p(\Rd,\mu_s)$.
 If $\omega\in \mathfrak{A}_{p_1}$ we have
\begin{align*}
\|G(t-1,s)f-m_s(f)\|_{L^{p_1}(\Rd,\mu_{t-1})}\le &M_{p_1,\omega}e^{\omega(t-s-1)}\|f\|_{L^{p_1}(\Rd,\mu_s)}\\
\le &M_{p_1,\omega}e^{\omega(t-s-1)}\|f\|_{L^p(\Rd,\mu_s)},
\end{align*}
for some positive constant $M_{p_1,\omega}$, independent of $f$. Then,
\begin{eqnarray*}
\|G(t,s)f-m_s(f)\|_{L^p(\Rd,\mu_t)}
\le M_{p_1,\omega}e^{\omega(t-s-1)}\|f\|_{L^p(\Rd,\mu_s)},
\end{eqnarray*}
so that $\mathfrak{A}_{p_1}\subset \mathfrak{A}_p$.

On the other hand, for   $t>s+1$,  $\omega\in \mathfrak{A}_p$ and   $f\in L^p(\Rd,\mu_s)$,  using again \eqref{hyper-estim}, we get
\begin{align*}
\|G(t,s)f-m_s(f)\|_{L^{p_1}(\Rd,\mu_t)}
\le &\|G(t,s)f-m_s(f)\|_{L^p(\Rd,\mu_t)}\\
= &\|G(t,s+1)G(s+1,s)f-m_s(f)\|_{L^p(\Rd,\mu_t)}\\
= &\|G(t,s+1)G(s+1,s)f-m_{s+1}(G(s+1,s)f)\|_{L^p(\Rd,\mu_t)}\\
\le &M_{p,\omega}e^{\omega(t-s-1)}\|G(s+1,s)f\|_{L^p(\Rd,\mu_{s+1})}\\
\le &M_{p,\omega}e^{\omega(t-s-1)}\|f\|_{L^{p_1}(\Rd,\mu_s)},
\end{align*}
for some positive constant $M_{p,\omega}$.
As above, we conclude that $\omega\in\mathfrak{A}_{p_1}$ so that $\mathfrak{A}_p\subset \mathfrak{A}_{p_1}$.

We have thus proved that $\mathfrak{A}_p=\mathfrak{A}_{p_1}$ for any $p\in (p_1,p_2]$. Starting from $p_2$, the same arguments yield  ${\mathfrak A}_p={\mathfrak A}_{p_2}={\mathfrak A}_{p_1}$ for any $p\in [p_2,p_3]$,
where $p_3 = e^{2\eta_0|r_0|\Lambda^{-1}}(p_2-1)+1$. Note that the sequence defined recursively by $p_1>1$, $p_{k+1} = e^{2\eta_0|r_0|\Lambda^{-1}}(p_{k}-1)+1$ has limit $+\infty$ as $k\to +\infty$. Hence,
 iterating this argument we
obtain $\mathfrak{A}_p=\mathfrak{A}_{p_1}$ for any $p\in [p_1,+\infty)$. Since $p_1\in (1,+\infty)$ is arbitrary, ${\mathfrak A}_p$ is independent of $p\in (1,+\infty)$.

In a similar way we   prove that $\mathfrak{B}_p$ is independent of $p\in (1,+\infty)$. Indeed, let $p_1$, $p_2$ and $p$ be as above. If $\omega \in \mathfrak{B}_{p_1}$ and $t-s\geq 2$,
using \eqref{grad_est_punt} and
\eqref{hyper-estim} we get, for $f\in C^1_b(\R^d)$,
\begin{align*}
\|\,|\nabla_xG(t,s)f|\,\|_{L^p(\Rd,\mu_t)}=& \|\,|\nabla_xG(t,t-1)G(t-1,s)f\|_{L^p(\Rd,\mu_t)}\\
\le & \|G(t,t-1)|\nabla_xG(t-1,s)f|\,\|_{L^p(\Rd,\mu_t)}\\
\le &\|\,|\nabla_xG(t-1,s)f|\,\|_{L^{p_1}(\Rd,\mu_{t-1})} \\
\le & N_{p_1,\omega}e^{\omega (t-s-1)}\|f\|_{L^{p_1}(\Rd,\mu_s)}\\
\le & N_{p_1,\omega}e^{\omega (t-s-1)}\|f\|_{L^p(\Rd,\mu_s)},
\end{align*}
for some positive constant $N_{p_1,\omega}$, independent of $f$, $s$ and $t$. Since $ C^1_b(\R^d)$ is dense in $L^p(\Rd,\mu_s)$ the estimate holds for every $f\in L^p(\Rd,\mu_s)$.
Similarly, if $\omega\in \mathfrak{B}_p$ and  $t-s\geq 2$,
then
\begin{align*}
\|\,|\nabla_xG(t,s)f|\,\|_{L^{p_1}(\Rd,\mu_t)}\le & \|\,|\nabla_xG(t,s)f|\,\|_{L^p(\Rd,\mu_t)}\\
= & \|\,|\nabla_xG(t,s+1)G(s+1,s)f|\,\|_{L^p(\Rd,\mu_t)}\\
\le & N_{p,\omega}e^{\omega(t-s-1)}\|G(s+1,s)f\|_{L^p(\Rd,\mu_{s+1})}\\
\le & N_{p,\omega}e^{\omega(t-s-1)}\|f\|_{L^{p_1}(\Rd,\mu_{s})},
\end{align*}
for  some positive constant $N_{p,\omega}$, independent of $f\in L^{p_1}(\Rd,\mu_{s})$, $s$ and $t$.

Then,  $\mathfrak{B}_p=\mathfrak{B}_{p_1}$ for any $p\in (p_1,p_2]$.
Iterating this argument as above, we conclude that  $\mathfrak{B}_p$ is independent of
$p\in (1,+\infty)$.

\emph{Step 2.} Let us prove that $\mathfrak{A}_2=\mathfrak{B}_2$. Fix $\omega\in \mathfrak{A}_2$,
$s,t\in I$, with $t-s\ge 1$, and $f\in L^2(\R^d,\mu_s)$ with $m_s(f)=0$.
Splitting $G(t,s)f=G(t,t-1)G(t-1,s)f$ and using \eqref{grad_est_norm_p} we get
\begin{align*}
\|\,|\nabla_x G(t,s)f|\,\|_{L^2(\R^d,\mu_t)}=&\|\,|\nabla_x G(t,t-1)G(t-1,s)f|\,\|_{L^2(\R^d,\mu_t)}\\
\le & c_2\|G(t-1,s)f\|_{L^2(\R^d,\mu_{t-1})}\\
\le & c_2 M_{2,\omega}e^{\omega(t-s)}\|f\|_{L^2(\R^d,\mu_s)}.
\end{align*}
If $m_s(f)\neq 0$ it suffices to apply the above estimate to   $f-m_s(f)$ and recall that
$\|f-m_s(f)\|_{L^2(\Rd,\mu_s)}\le 2\|f\|_{L^2(\Rd,\mu_s)}$.
Hence, $\omega\in \mathfrak{B}_2$, so that $\mathfrak{A}_2\subset\mathfrak{B}_2$.

Conversely, fix $\omega\in \mathfrak{B}_2$, $s,t\in I$, with $t-s\ge 1$ and $f\in L^2(\R^d,\mu_s)$.
Applying the Poincar\'e inequality \eqref{Poincare} (with $\mu_s$ and $f$   replaced by $\mu_t$ and $G(t,s)f$, respectively) and
observing that $m_t(G(t,s)f)=m_s(f)$, we get
\begin{align}
\|G(t,s)f-m_s(f)\|_{L^2(\R^d,\mu_t)}=&\|G(t,s)f- m_t(G(t,s)f)\|_{L^2(\R^d,\mu_t)}\notag\\
\le &C_2\|\,|\nabla_x G(t,s)f|\,\|_{L^2(\R^d,\mu_t)}\notag\\
\le &C_2N_{2,\omega}e^{\omega(t-s)}\|f\|_{L^2(\R^d,\mu_s)}.
\label{Ap-Bp}
\end{align}
If $t-s<1$ estimate \eqref{Ap-Bp} (with the constant $C_2N_{2,\omega}$ possibly replaced by a larger one) is a direct consequence of \eqref{contr}.
Then, $\omega\in {\mathfrak A}_2$ and the set equality
$\mathfrak{A}_2=\mathfrak{B}_2$ follows.
\end{proof}

An important consequence of Theorem \ref{thm-Ap-Bp} is an estimate for the exponential decay rate
of $G(t,s)f$ to $m_s(f)$.

\begin{coro}
\label{cor:compas}
For every $p>1$ there exists $C_p>0$ such that
\begin{eqnarray*}
\|G(t,s)f-m_s(f)\|_{L^p(\R^d,\mu_t)} \leq C_pe^{r_0(t-s)}\|f\|_{L^p(\R^d,\mu_s)}, \quad t>s\in I, \;f\in L^p(\R^d,\mu_s).
\end{eqnarray*}
\end{coro}
\begin{proof}
It is sufficient to apply Proposition
\ref{cor:decay} and Theorem \ref{thm-Ap-Bp}.
\end{proof}

\begin{rmk}
\rm{Arguing as in the proof of Theorem \ref{thm-Ap-Bp} it   is possible also to prove that $\mathfrak{A}_p=\mathfrak{C}_p$ for any $p\in (1,+\infty)$, where
\begin{align*}
\mathfrak{C}_p=\big\{&\omega\in\R: \|\,|\nabla_x G(t,s)f|\,\|_{L^p(\R^d,\mu_t)}\le \Theta_{p,\omega}e^{\omega ( t-s)}\|f\|_{W^{1,p}(\R^d,\mu_s)},\\
&{\rm for~any}\;s,t\in I,\, t>s,\;\,f\in W^{1,p}(\R^d,\mu_s)\;\,{\rm and~some}~\,\Theta_{p,\omega}>0 \big\}.
\end{align*}
}
\end{rmk}

\begin{rmk}
\rm{For $p=1$ the sets $\mathfrak{A}_1$ and $\mathfrak{B}_1$ may not coincide, even in the autonomous case. For instance, in the case of the Ornstein-Uhlenbeck operator $({\mathcal A}\zeta)(x) := \zeta''(x) - x\zeta'(x)$ we have $\mu_t(dx)=(2\pi)^{-1/2}e^{-x^2/2}dx$ for every $t$, and every $\lambda<0$ is an eigenvalue of the realization of ${\mathcal A}$ in $L^1( \R, \mu)$ as shown in \cite{MPP}. This implies that $\mathfrak{A}_1$ cannot contain negative numbers, so that $\mathfrak{A}_1=[0, +\infty)$. On the other hand, in this case $r_0= -1 \in \mathfrak{B}_1$ by Proposition \ref{cor:decay}. }
\end{rmk}

\section{The case of convergent coefficients as $t\to +\infty$}
\label{sect-asympt-meas}
Besides Hypotheses \ref{hyp1}, we assume here the following additional conditions.
\begin{hyp}
\label{hyp2}
\begin{enumerate}[\rm (i)]
\item
For any $r>0$ and some $($and hence any$)$ $t_0\in I$,
\begin{eqnarray*}
\sup_{(t,x)\in (t_0,+\infty)\times B(0,r)}|b(t,x)|<+\infty;
\end{eqnarray*}
\item
there exists a matrix $Q\in\mathcal{L}(\Rd)$ such that
\begin{eqnarray*}
\lim_{t\to +\infty}Q(t)=Q,
\end{eqnarray*}
in $\mathcal{L}(\Rd)$;
\item
there exist functions $b_j\in C^{\alpha}_{\rm loc}(\Rd)$ $(j=1,\ldots,d)$ such that
\begin{eqnarray*}
\lim_{t\to +\infty}b_j(t,x)=b_j(x),\qquad\;\,x\in\Rd,\;\,j=1,\ldots,d.
\end{eqnarray*}
\end{enumerate}
\end{hyp}

Let us consider the elliptic operator $\mathcal{A}$ defined on smooth functions $\zeta$ by
\begin{eqnarray*}
\mathcal{A}\zeta(x)=\sum_{i,j=1}^dq_{ij}D_{ij}\zeta(x)+\sum_{j=1}^db_j(x)D_j\zeta(x),\qquad\;\,x\in\Rd,
\end{eqnarray*}
Due to Hypotheses \ref{hyp1}(ii)-(iii), the operator $\mathcal{A}$ is uniformly elliptic. Moreover,
\begin{equation}
(\mathcal{A}\varphi)(x)\le a-c\varphi(x),\qquad\;\,x\in\Rd.
\label{khas}
\end{equation}
It is well known that, under Hypotheses \ref{hyp2}, there exists a Markov semigroup $T(t)$ associated to the operator $\mathcal{A}$ in $C_b(\Rd)$ (see \cite[Sect. 4]{MPW}).
For any $f\in C_b(\Rd)$ $u=T(\cdot)f$ is the unique solution to the Cauchy problem
\begin{eqnarray*}
\left\{
\begin{array}{ll}
D_tu(t,x)=\mathcal{A}u(t,x), & (t,x)\in (0,+\infty)\times\Rd,\\[1mm]
u(0,x)=f(x), &x\in\Rd,
\end{array}
\right.
\end{eqnarray*}
which belongs to $C_b([0,+\infty)\times\Rd)\cap C^{1,2}((0,+\infty)\times\Rd)$. Moreover, for any $t>0$, $T(t)$ is a contraction in $C_b(\Rd)$.

Condition \eqref{khas} and the Khas'minskii theorem yield the existence of a unique invariant measure
of the semigroup $T(t)$, i.e., a probability measure $\mu$ such that
\begin{eqnarray*}
\int_{\Rd}T(t)f\,\mu(dx)=\int_{\Rd}f\,\mu(dx),
\end{eqnarray*}
for any $f\in C_b(\Rd)$ (see \cite[Sect. 6]{MPW}).

\begin{thm}
\label{thm-limit}
For any $f\in C_b(\Rd)$
\begin{eqnarray*}
\lim_{t\to +\infty}\int_{\Rd}f\mu_t(dx)=\int_{\Rd}f\mu(dx),
\end{eqnarray*}
i.e., $\mu_t$ weakly$^*$ converges to $\mu$ as $t\to +\infty$.
\end{thm}

\begin{proof}
Since the evolution system of measures $\{\mu_t: t\in I\}$ is tight, it admits weak$^*$ limits as $t\to +\infty$.
We will prove that $\mu$ is its only weak$^*$ limit.
For this purpose, let $(s_n)$ and $\nu_0$ be, respectively, a sequence diverging to $+\infty$ and a probability measure such that
\begin{equation}
\lim_{n\to +\infty}\int_{\Rd}f\mu_{s_n}(dx)=\int_{\Rd}f\,\nu_0(dx),\qquad\;\,f\in C_b(\Rd).
\label{mart-6}
\end{equation}
We split the proof into four steps.

{\em Step 1.} Here, we prove that, for any $f\in C^{\infty}_c(\Rd)$, the sequence
$(G(\cdot+s_n,s_n)f)$ converges to $T(\cdot)f$ as $n\to +\infty$, locally uniformly in $[0,+\infty)\times\Rd$.
To this aim, we observe that the function $v_n=G(\cdot+s_n,s_n)f$ belongs to $C_b([0,+\infty)\times\Rd)\cap C^{1,2}([0,+\infty)\times\Rd)$
and solves the Cauchy problem
\begin{eqnarray*}
\left\{
\begin{array}{ll}
D_tv_n(t,x)=\mathcal{A}(t+s_n)v_n(t,x), & (t,x)\in [0,+\infty)\times\Rd,\\[1mm]
v_n(0,x)=f(x), &x\in\Rd.
\end{array}
\right.
\end{eqnarray*}
Due to Hypotheses \ref{hyp2} and to the classical Schauder estimates, for any $r,T>0$, there exists a positive
constant $C$ (which is independent of $n$) such that
\begin{eqnarray*}
\|v_n\|_{C^{1+\alpha/2,2+\alpha}([0,T]\times B(0,r))}\le C\|f\|_{C^{2+\alpha}_b(\Rd)}.
\end{eqnarray*}
Arzel\`a-Ascoli theorem and a diagonal argument show that there exist a function $u\in C_b([0,+\infty)\times\Rd)\cap C^{1+\alpha/2,2+\alpha}_{\rm loc}([0,+\infty)\times\Rd)$ and a subsequence $(s_{n_k})\subset (s_n)$ such that $v_{n_k}\to u$ in $C^{1,2}([0,T]\times B(0,r))$ for any $T,r>0$, and
$D_tu=\mathcal{A}u$ in $[0,+\infty)\times\Rd$. Moreover, $u(0,\cdot)=f$ in $\Rd$ since $u_n(0,\cdot)=f$
for any $n\in\N$.
Hence, $u=T(\cdot)f$. Actually, all the sequence $(G(\cdot+s_n,s_n)f)$ converges locally uniformly in $[0,+\infty)\to
\R^d$ to $T(\cdot)f$. Indeed, the previous argument shows that any subsequence $(s'_n)\subset (s_n)$ admits a subsequence
$(s'_{n_k})$ such that $G(\cdot+s'_{n_k},s'_{n_k})$ converges to $T(\cdot)f$ in $C^{1,2}([0,T]\times B(0,r))$ for
any $r,T>0$, and this is enough to infer that the sequence $(G(\cdot+s_n,s_n))$ converges to $T(\cdot)f$ locally uniformly
in $[0,+\infty)\times\Rd$.

{\em Step 2.} For any $k\in\N$, the system of measures $\{\mu_{k+s_n}: n\in\N\}$ is tight. Hence, by the Prokhorov theorem,
there exist a subsequence $(s_n^k)$ of $(s_n)$ and a probability measure $\nu_k$ such that
\begin{eqnarray*}
\lim_{n\to +\infty}\int_{\Rd}f\mu_{k+s_n^k}(dx)=\int_{\Rd}f\nu_k(dx),\qquad\;\,f\in C_b(\Rd).
\end{eqnarray*}
Using again a diagonal argument, we can extract a subsequence $(t_n)$ of $(s_n)$ such that
\begin{equation}
\lim_{n\to +\infty}\int_{\Rd}f\mu_{k+t_n}(dx)=\int_{\Rd}f\nu_k(dx),\qquad\;\,f\in C_b(\Rd).
\label{mart-sette}
\end{equation}
for any $k\in\N$.

Since $\{\mu_t: t\in I\}$ is tight, the set of measures $\{\nu_k: k\in\N\cup\{0\}\}$ is tight too.
Indeed, fix $\varepsilon>0$ and let $R_0>0$ be such that
$\mu_t(B(0,R_0))\ge 1-\varepsilon$ for any $t\in I$.
Let $\psi\in C_b(\Rd)$ satisfy $\chi_{B(0,R_0)}\le\psi\le\chi_{B(0,R_0+1)}$.
Then,
\begin{eqnarray*}
\int_{\Rd}\psi\mu_{k+s_n}(dx)\ge\int_{B(0,R_0)}\one\mu_{k+s_n}(dx)=\mu_{k+s_n}(B(0,R_0))\ge 1-\varepsilon,
\end{eqnarray*}
for any $k\in\N\cup\{0\}$. Letting $n\to +\infty$ gives
\begin{eqnarray*}
1-\varepsilon\le \int_{\Rd}\psi\nu_k(dx)\le \int_{B(0,R_0+1)}\one\nu_k(dx)
\end{eqnarray*}
and we thus conclude that
$\nu_k(B(0,R_0+1))\ge 1-\varepsilon$ for any $k\in\N$, showing that the set $\{\nu_k: k\in\N\}$ is tight.

{\em Step 3.} Here, we prove that
\begin{equation}
\int_{\Rd}T(k)f\nu_k(dx)=\int_{\Rd}f\nu_0(dx),
\label{inv-quasi}
\end{equation}
for any $f\in C^{\infty}_c(\Rd)$. Using \eqref{invariant-0} and \eqref{mart-6} we get
\begin{equation}
\lim_{n\to +\infty}\int_{\Rd}G(k+t_n,t_n)f\mu_{k+t_n}(dx)
=\lim_{n\to +\infty}\int_{\Rd}f\mu_{t_n}(dx)=\int_{\Rd}f\nu_0(dx),
\label{limite-0}
\end{equation}
for any $k\in\N$.
We claim that the left-hand side of \eqref{limite-0} equals $\int_{\Rd}T(k)f\nu_k(dx)$.
Indeed,
\begin{align*}
&\int_{\Rd}G(k+t_n,t_n)f\mu_{k+t_n}(dx)-
\int_{\Rd}T(k)f\nu_k(dx)\\
=&\int_{\Rd}\left (G(k+t_n,t_n)f-T(k)f\right )\mu_{k+t_n}(dx)\\
&+\left (\int_{\Rd}T(k)f\mu_{k+t_n}(dx)-\int_{\Rd}T(k)f\nu_k(dx)\right )\\
=:&I_{1,n}+I_{2,n}.
\end{align*}
By \eqref{mart-sette} $I_{2,n}$ tends to $0$ as $n\to +\infty$. To prove that also $I_{1,n}$
vanishes as $n\to +\infty$, we fix $\varepsilon>0$ and $R>0$ such that
$\mu_t(\Rd\setminus B(0,R))\le\varepsilon$ for any $t\in I$. Then, we estimate
\begin{align*}
|I_{1,n}|\le & \int_{B(0,R)}|G(k+t_n,t_n)f-T(k)f|\mu_{k+t_n}(dx)\\
&+\int_{\Rd\setminus B(0,R)}|G(k+t_n,t_n)f-T(k)f|\mu_{k+t_n}(dx)\\
\le &\|G(k+t_n,t_n)f-T(k)f\|_{L^{\infty}(B(0,R))}+
2\|f\|_{\infty}\varepsilon,
\end{align*}
since both $G(k+t_n,t_n)f$ and $T(k)f$ are contractions in $C_b(\Rd)$. By Step 1 
\begin{eqnarray*}
\limsup_{n\to +\infty}|I_{1,n}|\le\varepsilon
\end{eqnarray*}
and, since $\varepsilon>0$ is arbitrary, we conclude that $I_{1,n}$ tends to $0$
as $n\to +\infty$.
Formula \eqref{inv-quasi} follows.

{\em Step 4.} Here, using formula \eqref{inv-quasi} we conclude the proof.
We remark that estimate \eqref{grad_est_punt}, with $p=1$, and Step 1 imply
$\|\,|\nabla_x T(t)f|\,\|_{\infty}\le e^{-r_0t}\|\,|\nabla f|\,\|_{\infty}$ for every $t>0$. Arguing
as in the proof of Lemma \ref{norma_p} we obtain that $T(k)f$ converges to $\int_{\Rd}f\mu(dx)$
locally uniformly in $\Rd$, for any $f\in C^{\infty}_c(\Rd)$.

Now, using the same argument of Step 3, we
show that
\begin{equation}
\lim_{k\to +\infty}\int_{\Rd}T(k)f\nu_k(dx)=\int_{\Rd}f\mu(dx).
\label{limite-1}
\end{equation}
Since $\{\nu_k: k\in\N\}$ is a tight system of probability
measures, for any $\varepsilon>0$ there exists $R>0$ such that $\nu_k(\Rd\setminus B(0,R))\le\varepsilon$ for
any $k\in\N$. Hence,
\begin{align*}
\left |\int_{\Rd}T(k)f\nu_k(dx)-\int_{\Rd}f\mu(dx)\right |
\le &\int_{\Rd}\left |T(k)f-\int_{\Rd}f\mu(dx)\right |\nu_k(dx)\\
\le &\int_{B(0,R)}\left |T(k)f-\int_{\Rd}f\mu(dx)\right |\nu_k(dx)\\
&+\int_{\Rd\setminus B(0,R)}\left |T(k)f-\int_{\Rd}f\mu(dx)\right |\nu_k(dx)\\
\le &\left\|T(k)f-\int_{\Rd}f\mu(dx)\right\|_{L^{\infty}(B(0,R))}
+2\|f\|_{\infty}\varepsilon,
\end{align*}
which implies that
\begin{eqnarray*}
\limsup_{k\to +\infty}
\left |\int_{\Rd}T(k)f\nu_k(dx)-\int_{\Rd}f\mu(dx)\right |
\le\varepsilon
\end{eqnarray*}
and \eqref{limite-1} follows.

Formulae \eqref{inv-quasi} and \eqref{limite-1} yield
\begin{eqnarray*}
\int_{\Rd}f\mu(dx)=\int_{\Rd}f\nu_0(dx),\qquad\;\,f\in C^{\infty}_c(\Rd).
\end{eqnarray*}
This shows that $\nu_0=\mu$ and completes the proof.
\end{proof}

\begin{rmk}
{\rm Arguing as in the proof of Theorem \ref{thm-limit}, we can prove that, if $\omega\in \mathfrak{A}_p$, then 
there exists $M>0$ such that 
\begin{eqnarray*}
\left\|T(t)f-\int_{\Rd}f\mu(dx)\right\|_{L^p(\Rd,\mu)}\le Me^{\omega t}\|f\|_{L^p(\Rd,\mu)},\qquad\;\,f\in L^p(\Rd,\mu).
\end{eqnarray*}
}
\end{rmk}

\begin{rmk}
{\rm 
In the case of $T$-time periodic coefficients considered in \cite{LorLunZam10}, the tight evolution system of
measures $\{\mu_t: t\in I\}$ is $T$-periodic, namely $\mu_{t+T}=\mu_t$ for any $t\in\R$, and
$\{\mu_t: t\in I\}$ coincides with the set of its weak$^*$ limit measures. Indeed, if a sequence $\mu_{t_n}$ weakly$^*$ converges to $\mu$ as $t_n\to +\infty$, then setting $s_n=t_n-[t_n/T]T\in [0,T)$ (where $[\,\cdot\,]$ denotes the integer part)
a subsequence $(s_{n_k})$ of $(s_n)$ converges to some $s_0\in [0,T]$. Since the function $s\mapsto\int_{\Rd}f\mu_s(dx)$ is
continuous in $\R$ for every $f\in C_b(\Rd)$, we obtain 
\begin{eqnarray*}
\int_{\Rd}f\mu_{t_{n_k}}(dx)=\int_{\Rd}f\mu_{s_{n_k}}(dx)\to \int_{\Rd}f\mu_{s_0}(dx),
\end{eqnarray*}
as $k\to +\infty$, for every $f\in C_b(\Rd)$. Hence, $\mu=\mu_{s_0}$. Conversely, since each measure $\mu_t$ is the weak$^*$ limit
of the sequence of measure $(\mu_{t+nT})$, the claim follows.
}
\end{rmk}


\section{Comments and examples}
\label{examples}


The key tool of our analysis is estimate \eqref{grad_est_punt}, and of course the existence of a tight evolution system of measures for $G(t,s)$.  With the noteworthy exception of time depending Ornstein-Uhlenbeck operators (see Subsection 6.1), Hypotheses \ref{hyp1}(iii) and \ref{hyp1}(iv) are quite natural conditions that guarantee the validity of estimate \eqref{grad_est_punt} and the existence of a tight evolution system of measures, respectively.

Still assuming that  Hypotheses \ref{hyp1}(i)-(iii) hold,  we could avoid Hypotheses \ref{hyp1}(iv), replacing it by  its consequences

\begin{align*}
&\exists \kappa <0: \;|(\nabla_x G(t,s)f)(x)| \leq e^{\kappa (t-s)}(G(t,s)|\nabla f|)(x),\qquad\;\,
I\ni s<t,\;\,x\in\R^d,\\
&\|\, |\nabla_x G(t,s)f|\,\|_{L^p(\Rd,\mu_t)}\leq c_p(t-s)^{-1/2}\Vert f\Vert_{L^p(\Rd,\mu_s)},\qquad\;\, s< t \le s+1,
\end{align*}
plus some control on $\langle b(t,x), x\rangle$, such as $\langle b(t,x), x\rangle \leq C_{a,b} (1+|x|^2) \varphi(x)$ for $a\leq t\leq b$, $x\in \R^d$, to estimate the integrals
\begin{eqnarray*}
\int_{\R^d} G(t,s)f^p \log (G(t,s)f^p) {\mathcal A}(t)\theta_n \, \mu_t(dx), \quad \int_{\R^d} G(t,s)f^p   {\mathcal A}(t)\theta_n \, \mu_t(dx),
\end{eqnarray*}
that arise in the proof of Theorems \ref{thm_log_sob} and \ref{thm_fin}.

\subsection{Nonautonomous Ornstein-Uhlenbeck operators}
\label{sect-appl}
 Time depending  Orn\-stein-Uhlenbeck operators,
\begin{equation}\label{OU}
(\mathcal{L}(t)\zeta)(x)= \frac{1}{2}\textrm{Tr}(B(t)B^*(t)D^2\zeta(x))
+\langle A(t)x+g(t), \nabla\zeta(x) \rangle,\qquad\;\, t\in\R,\;\, x \in \Rd,
\end{equation}
have been studied in   \cite{GeisLun08,GeisLun09}. In fact, in these papers backward Cauchy problems were considered and the evolution operator was backward, however all the statements may be easily rewritten as statements for forward evolution equations and evolution operators.

The assumptions  are the following:
$B,A:\R\to\mathcal{L}(\Rd)$ and $g:\R \to \Rd$ are continuous and bounded,   $Q(t):= B(t)(B(t))^*/2$ satisfies the uniform ellipticity condition \eqref{elliptic}, and the norm of the evolution operator $U(t,s)$ in $\Rd$  associated  to the equation $u'(t) = -A(t)u(t)$ decays exponentially as $t-s\to -\infty$. More precisely,  for any $s\in\R$,
$U(\cdot,s)$ is the unique solution to the Cauchy problem
\begin{eqnarray*}
\left\{
\begin{array}{ll}
D_tU(t,s)=-A(t)U(t,s), & t\in\R,\\[1mm]
U(s,s)=I,
\end{array}
\right.
\end{eqnarray*}
where $I$ denotes the identity matrix, and the decay estimate is
\begin{equation}\label{espone_U}
\omega_0(U)=\inf\{\omega\in \R: \exists M_\omega\ge 1\,\, \textrm{such that} \,\Vert U(t,s)\Vert\le M_\omega e^{\omega(s-t)},\, t\le s\}<0.
\end{equation}
Then, it is possible to write down explicitly both $G(t,s)$ and all the systems of evolution measures for $G(t,s)$. Among them, the unique tight system of evolution measures consists of Gaussian measures $\mu_t$ given by
\begin{equation}\label{mut}
\mu_t(dx) = (2\pi)^{-\frac{d}{2}}(\det Q_t)^{- \frac{1}{2}}
\,e^{-\frac{1}{2}\langle Q_t^{-1}(x-g_t),x-g_t\rangle}dx,\qquad\;\, t\in\R,
\end{equation}
where
\begin{align*}
Q_t &=\int_t^{+\infty}U(t,\xi)B(\xi)B^*(\xi)U^*(t,\xi)d\xi,\qquad\;\, t\in\R,\\
g_t &=\int_t^{+\infty}U(t,\xi)g(\xi)d\xi,\qquad\;\, t\in\R.
\end{align*}

So, in this case Hypothesis \ref{hyp1}(iii) is not needed, since the tight evolution system of measures is explicit.
Assumption \eqref{espone_U} is not equivalent to our Hypothesis \ref{hyp1}(iv). For instance, taking $d=2$, $A(t) = \left( \begin{array}{cc} -1 & 0\\ 2 & -1\end{array}\right)$, we get easily $\omega_0(U)=-1$, but \eqref{b} is not satisfied by any $r_0<0$.

However,
the proof of Theorem \ref{thm-Ap-Bp} relies on the hypercontractivity estimates,  on the Poincar\'e inequality for $p=2$, that holds for the above Gaussian measures   with constants independent of $t$, and on the gradient estimates \eqref{grad_est_punt} and \eqref{grad_est_norm_p}.

Hypercontractivity and estimates \eqref{grad_est_punt}, \eqref{grad_est_norm_p}  were proved in \cite[Thm. 3.3]{GeisLun09},
\cite[Lemma 3.3]{GeisLun08}. Thus, we
 can   apply Theorem \ref{thm-Ap-Bp} and conclude that $\mathfrak{A}_p=\mathfrak{B}_p$ for any $p>1$.
Again \cite[Lemma 3.3]{GeisLun08} shows that $\mathfrak{B}_p\supset (\omega_0(U),+\infty)$ so that
\begin{equation}
\Vert G(t,s)f-m_s(f)\Vert_{L^p(\Rd,\mu_t)}\leq M_{p,\omega}e^{\omega(t-s)}\Vert f \Vert_{L^p(\Rd,\mu_s)},\qquad\;\,t>s,
\label{decad-O-U}
\end{equation}
for every $\omega \in (\omega_0(U),0)$, $f \in L^p(\Rd,\mu_s)$, $p > 1$ and some positive constant $M_{p,\omega}$.
Estimate \eqref{decad-O-U} improves the decay estimate
obtained for $p=2$ in \cite[Thm. 2.17]{GeisLun09}, and it answers positively to  the conjecture raised in \cite{GeisLun09}.

\subsection{More general nonautonomous operators}
Here we exhibit some classes of nonautonomous elliptic operators that satisfy Hypothesis \ref{hyp1}(iii), since the other ones are easy to be checked.

Let $(\mathcal{A}(t))_{t \in I}$ be defined by \eqref{a(t)}, and assume that its coefficients $q_{ij}$ and $b_i $ satisfy the regularity and ellipticity assumptions of Hypotheses \ref{hyp1}(i)-(ii), and the dissipativity condition in Hypothesis \ref{hyp1}(iv). Moreover, we assume that   there exist three positive constants  $K_{1}$, $R$, $\beta >1$ such that
\begin{equation}\label{ex_1_1}
\langle b(t,x),x\rangle \leq - K_{1}|x|^{\beta},\qquad\;\,t\in I,\;\,x\in\Rd\setminus B(0,R),
\end{equation}
Then   the function
\begin{eqnarray*}
\varphi(x)=e^{\delta|x |^{\beta}}, \quad x\in \R^d,
\end{eqnarray*}
satisfies Hypothesis \ref{hyp1}(iii).
Indeed, for $t\in I$ and $|x| \geq R$ we have
\begin{align*}
(\mathcal{A}(t)\varphi)(x)=&  \delta \beta \varphi(x) \big[ (\delta \beta |x|^{2\beta -4} + (\beta -2) |x|^{\beta -4} )\langle Q(t)x,x\rangle\\
& + \textrm{Tr}(Q(t))|x|^{\beta - 2}  + \langle b(t,x),x\rangle |x|^{\beta - 2}\big]\\
\leq &  \delta \beta \varphi(x) [\delta \beta \Lambda |x|^{2\beta -2} + \Lambda (d+  \beta -2) |x|^{\beta -2} -  K_{1}|x|^{2\beta -2} )]\\
=: & g_1(x) \varphi(x).
\end{align*}
If $\delta$ is small enough (i.e., $\delta<K_1/(\beta\Lambda)$),  then $\lim_{|x|\to + \infty} g_1(x) = -\infty$.
Then there exist   $a$, $c>0$ such that ${\mathcal A}(t)\varphi\le a-c\varphi$
for $t\in I$.

If \eqref{ex_1_1} is replaced by
\begin{equation}\label{ex_1_2}
\exists \overline{y}\in \R^d:\; \sup_{t\in I} |b(t,\overline{y})| =: K_{2} <+\infty,
\end{equation}
then for sufficiently small $\delta >0$  the function
\begin{eqnarray*}
\varphi(x)=e^{\delta|x- \overline{y}|^2}, \qquad\;\, x\in \R^d
\end{eqnarray*}
satisfies Hypothesis  \ref{hyp1}(iii). Indeed, taking into account that by \eqref{monotonia} we have $\langle b(t,x), x-  \overline{y}\rangle \leq
r_0|x-  \overline{y}|^2 + K_{2 }|x-  \overline{y}|$ for $t\in I$ and $x\in \R^d$, we obtain
\begin{align*}
(\mathcal{A}(t)\varphi)(x)=& (4\delta^2\langle Q(t)(x-  \overline{y}),(x-  \overline{y})\rangle + 2\delta\textrm{Tr}(Q(t)) + 2\delta \langle b(t,x),x-  \overline{y}\rangle) \varphi(x)
\\
\leq & (4\delta^2\Lambda |x-  \overline{y}|^2 +   2\delta d\Lambda + 2\delta (r_0|x-  \overline{y}|^2 + K_{2}|x-  \overline{y}|))\varphi(x)
\\
  =: & g_2(x) \varphi(x).
  \end{align*}
If $\delta$ is small enough ($\delta < |r_0|/2\Lambda$), then $\lim_{|x|\to +\infty} g_2(x) = -\infty$.   Then, there exist two positive constants $a$, $c$ such that ${\mathcal A}(t)\varphi\le a-c\varphi$
for $t\in I$.

Note that condition \eqref{ex_1_2}
is satisfied for every $\overline{y}\in \R^d$ in the case of time periodic coefficients considered in \cite{LorLunZam10}.

\end{document}